\documentclass{amsart}
\usepackage{amssymb}
\usepackage{amsfonts}

\newtheorem{theorem}{Theorem}[section]
\theoremstyle{plain}

\newtheorem{corollary}[theorem]{Corollary}

\newtheorem{definition}[theorem]{Definition}
\newtheorem{example}[theorem]{Example}

\newtheorem{lemma}[theorem]{Lemma}

\newtheorem{proposition}[theorem]{Proposition}
\theoremstyle{remark}
\newtheorem{remark}[theorem]{Remark}

\numberwithin{equation}{section}
\input{tcilatex}

\begin{document}
\title[The Jacobson radical of an evolution algebra]{The Jacobson radical of
an evolution algebra}
\author{M. Victoria Velasco}
\address{Departamento de An\'{a}lisis Matem\'{a}tico, Universidad de
Granada, 18071 Granada, Spain.}
\email{vvelasco@ugr.es}
\urladdr{}
\thanks{}
\date{December 2016}
\thanks{This research has been funded by Project MTM2016-76327-C3-2-P of the
Spanish Ministerio of Econom\'{\i}a and Competitividad, and the Research
Group FQM 199 of the Junta de Andaluc\'{\i}a, both of them partially
supported with Fondos Feder.}
\dedicatory{}

\begin{abstract}
In this paper we characterize the maximal modular ideals of an evolution
algebra $A\,\ $in order to describe its Jacobson radical, \ $Rad(A).$ We
characterize semisimple evolution algebras (i.e. those such that $%
Rad(A)=\{0\}$)as well as radical ones. We introduce two elemental notions of
spectrum of an element $a$ in an evolution algebra $A$, namely the spectrum $%
\sigma ^{A}(a)$ and the m-spectrum $\sigma _{m}^{A}(a)$ (they coincide for
associative algebras, but in general $\sigma ^{A}(a)\subseteq \sigma
_{m}^{A}(a),$ and we show examples where the inclusion is strict). We prove
that they are non-empty and describe $\sigma ^{A}(a)$ and $\sigma _{m}^{A}(a)
$ in terms of the eigenvalues of a suitable matrix related with the
structure constants  matrix of $A.$ We say $A$ is m-semisimple (respectively
spectrally semisimple) if zero is the unique \ ideal contained into the set
of $a$ in $A$ such that $\sigma _{m}^{A}(a)=\{0\}$ $\ $(respectively $\sigma
^{A}(a)=\{0\}$). In contrast to the associative case (where the notions of
semisimplicity, spectrally semisimplicty and m-semisimplicity are
equivalent)\ we show examples of m-semisimple evolution algebras $A$ that,
nevertheless, are radical algebras (i.e. $Rad(A)=A$). Also some theorems
about automatic continuity of homomorphisms will be considered.
\end{abstract}

\maketitle

\section{\protect\bigskip \label{radisemi}Introduction}

By an \textbf{algebra }we understand a linear space $A$ over the field $%
\mathbb{K}$ (where either $\mathbb{K}=\mathbb{R}$ or $\mathbb{C}$)\ provided
with a bilinear map $(a,b)\rightarrow ab,$ from $A\times A\rightarrow A,$
named the \textbf{multiplication of} $A.$ An algebra $A$ is said to be 
\textbf{associative }if $(ab)c=a(bc),$ for every $a,b,c\in A,$ and \textbf{\
commutative} if $ab=ba,$ for every $a,b\in A.$ Therefore, in contrast for
instance with \cite{Bonsall,Dales,Palmer,Rickart}, throughout this paper,
the product of an algebra does not need to be associative (or commutative),
unless this fact is explicitly specified. Moreover, we say that $e\in A$ is
a \textbf{unit }of $A$ if $ae=ea=a,$ for every $a\in A.$ Obviously such a
unit element is unique whenever it exists.

Relevant examples of algebras in this general meaning that we are
considering are \textbf{evolution algebras}. These algebras are very
meaningful in Genetics and its theory was founded in \cite{Tian}. There, it
is justified why evolution algebras have emerged to enlighten the study of
Non-Mendelian Genetics, which is the basic language of Molecular Biology. In
this pioneering monograph many connections of evolution algebras with other
mathematical fields (such as Graph Theory, Stochastic Processes, Group
Theory, Dynamic Systems, Mathematical Physics, etc.) are established,
pointing out some further research topics.

An \textbf{evolution algebra} is an algebra $A$ provided with a basis $%
B:=\{e_{i}:i\in \Lambda \}$ such that $e_{i}e_{j}=0,$\ whenever $i\neq j.$\
Such a basis $B$\ is named a \textbf{natural basis }of \emph{\ }$A.$ If $%
e_{i}^{2}=\sum_{k\in \Lambda }\omega _{ki}e_{k}$, then the \textbf{structure
matrix }is given by

\begin{equation*}
M_{B}(A)=(\omega _{ki})\in M_{\Lambda }(\mathbb{K}),
\end{equation*}%
where $M_{\Lambda }(\mathbb{K})$ denote the set of matrices in ${\mathbb{K}}%
^{\Lambda \times \Lambda }$ whose columns have a finite number of non-zero
entries. The structure matrix defines the product of $A,$ codifying also the
dynamic structure of $A.$

Recently, many different aspects of the theory of evolution algebras have
been developed. For instance, in \cite{Ca-Si-Ve} many algebraic properties
of evolution algebras (like simplicity, decomposability in direct sum of
ideals), or graphs associated to them, are deeply studied. \ The derivations
of evolution algebras have been analyzed in \cite{Tian,evo3,evo4}. In \cite%
{evo4}, the evolution algebras have been used to describe the inheritance of
a bisexual population and, in this setting, the existence of non-trivial
homomorphisms onto the sex differentiation algebra have been studied in \cite%
{evo14}. Algebraic notions such as nilpotency and solvability may be
interpreted biologically as the fact that some of the original gametes (or
generators) become extinct after certain generations, and these algebraic
properties have been studied in \cite{evo5,evo6,evo7,evo8,evo9,evo12,evo13}.
Moreover evolution algebras associated with function spaces defined by Gibbs
measures on a graph are considered in \cite{evo2}, to provide a natural
introduction of thermodynamics in the studying of several systems in
Biology, Physics and Mathematics. On the other hand, chains of evolution
algebras (i.e. dynamical systems the state of which at each given time is an
evolution algebra) are considered in \cite{evo1,evo10,evo11,evo15}.

The aim of this paper is to study the Jacobson radical of an evolution
algebra, as well as some notions of semisimplicity related to it. The
classical definition of Jacobson radical of an associative Banach algebras
was extended to the setting of non-associative algebras in \cite[Definition 4%
]{Ma-Ve2}. According to \cite{Ma-Ve2}, \ the \textbf{Jacobson radical} of a
commutative algebra is defined as the intersection of all maximal modular
ideals (in other words, the intersection of all primitive ideals in the
meaning of \cite[Definition 3]{Ma-Ve2}). A \textbf{modular ideal }of a
commutative algebra $A$ is an ideal $M\,\ $endowed with a modular unit, that
is $u\in A$ such that $a-au\in M,$ for every $a\in A.$ It is not difficult
to check that an ideal $M$ is modular if and only if $A/M$ (provided with
the canonical quotient product) is an algebra with a unit.

When $A$ has a unit it turns out that every ideal $I\,\ $of \thinspace $A$ \
is modular ideal (because the unit of $A$ is a modular unit for $I$). Since
the existence of a unit is someway relevant for our approach, Section 2 of
this paper is aimed to show that evolution algebras exceptionally have a
unit. More precisely in Proposition \ref{unidad} we prove that \textit{%
infinite-dimensional evolution algebras have not a unit. Moreover, a
finite-dimensional evolution algebra has a unit if, and only if, the matrix
structure of }$A$\textit{\ relative to an arbitrary natural basis }$B\,$%
\textit{\ is diagonal with non-zero entries. }This result enlightens
Proposition 1 in \cite[Section 3.1.2]{Tian}, where the finite-dimension of
the considered evolution algebra is implicitly assumed. Moreover, we\ obtain
that \textit{every non-trivial evolution algebra does not have a unit, and
its unitization is not an evolution algebra} (see Corollary \ref{unifinal}).

Modular ideals of an evolution algebra are characterized in Corollary \ref%
{coro}. As a consequence, in Corollary \ref{coroa}, we obtain that \textit{%
maximal modular ideals of an evolution algebra have codimension one.} More
precisely, if \textit{\ }$B:=\{e_{i}:i\in \Lambda \}$ is a natural basis of
an evolution algebra $A$ then, the set of \textbf{modular index} associated
to $\Lambda $ is defined $\ $as $\Lambda _{m}=\{$ $i\in \Lambda :$ $\omega
_{ii}\neq 0$ and $\omega _{ij}=0$ if $i\neq j\}$ (see Definition \ref%
{modularindex}) and, in Corollary \ref{moduin}, it is showed that $M$\textit{%
\ is a maximal modular ideal of }$A$ \textit{if, and only if, }$%
M=lin\{e_{i}:i\in \Lambda \backslash \{i_{0}\}\},$\textit{\ for some }$%
i_{0}\in \Lambda _{m},$ \ \textit{in which case }$\frac{1}{\omega
_{i_{0}i_{0}}}e_{i_{0}}$\textit{\ is a modular unit for} $M.$ Therefore, in
Corollary \ref{Jradical}, we obtain that \textit{the Jacobson radical of }$A$%
\textit{\ can be described as follows, }%
\begin{equation*}
Rad(A)=lin\{e_{i}:i\in \Lambda \backslash \Lambda _{m}\}.
\end{equation*}%
\textit{Consequently, }$A$\textit{\ is a radical algebra if and only if }$%
\Lambda _{m}=\emptyset .$\textit{\ Similarly }$A$\textit{\ is semisimple if
and only if }$\Lambda =\Lambda _{m}.$\textit{\ }Moreover, in Corollary \ref%
{nuevo}, we prove that $A/Rad(A)$\textit{\ is either }$\{0\}$ \textit{or} 
\textit{a non-zero trivial evolution algebra.}

In Proposition \ref{nilpo} we show that\textit{\ every quasi-invertible
ideal, and particularly every nilpotent ideal, of an evolution algebra }$A$ 
\textit{is contained in }$Rad(A).$ Nevertheless, in contrast with the
associative case, we provide an example of an evolution algebra containing
elements in its Jacobson radical that are not quasi-invertible (see Example %
\ref{quasi-inver}).

In Section \ref{espec}, we review the notion of spectrum of an element, $a,$
in a non- associative algebra, $A,$ by considering two definitions of
spectrum, namely $\sigma ^{A}(a)$ and \ $\sigma _{m}^{A}(a)$ (the spectrum
and the multiplicative spectrum of $a,$ respectively)$.$ More precisely, for
a complex algebra $A$ which a unit, $e,$ we define for every $a\in A,$

\begin{equation*}
\sigma ^{A}(a):=\{\lambda \in \mathbb{C}:a-\lambda e\text{ is not invertible}%
\},
\end{equation*}%
in the meaning that $a-\lambda e$ has not left or a right inverse, and on
the other hand we define

\begin{equation*}
\sigma _{m}^{A}(a):=\{\lambda \in \mathbb{C}:a-\lambda e\text{ is not
m-invertible}\},
\end{equation*}%
in the meaning that $L_{a}-\lambda I$ or $R_{a}-\lambda I$ are not bijective
(where \thinspace $L_{a}$ and $R_{a}$ denote, respectively, the left and\
right multiplication operator by $a$)$.$ It is clear that $\sigma
_{{}}^{A}(a)\subseteq \sigma _{m}^{A}(a)$ and also that for an associative
algebra $A$ we have $\sigma _{{}}^{A}(a)=\sigma _{m}^{A}(a)$. Nevertheless,
in Example \ref{estrict}, we show an element $a$ in an evolution algebra $A$
for which $\sigma _{{}}^{A}(a)$ is a proper subset of $\sigma _{m}^{A}(a)$.

On the other hand, for an evolution algebra $A$ and $a\in A,$ we prove that
the sets $\sigma ^{A}(a)$ and $\sigma _{m}^{A}(a)$ are non-empty and we
characterize both of them in Proposition \ref{propoa} (for trivial evolution
algebras), in Proposition \ref{propob} (for finite-dimensional evolution
algebras), and in Proposition \ref{propoc} (for the general case), according
to the eigenvalues of certain matrices related with a structure matrix.

As said in \cite[p. 189]{Palmer} (see also \cite[Theorem 4.3.6]{Dales}), in
the classical theory of Banach algebras, \textit{the radical of Jacobson of
an associative algebra }$A$ \textit{can be described as the largest ideal on
which the spectral radius of each element is identically zero. Consequently, 
}$A$ \textit{\ is semisimple if its radical is zero, that is if }$\{0\}$%
\textit{\ is the only ideal contained in the set of elements having spectral
radius equal to zero. }

Since we have two notions of spectrum for an element $a\in $ $A$, namely $%
\sigma _{{}}^{A}(a)$ and $\sigma _{m}^{A}(a),$ we define the corresponding
notions of spectral radius, $\rho $ and $\rho _{m},$ in Definition \ref%
{radius}. Thus, for an evolution algebra $A$ $\ $and $a\in A,$ we have that $%
\rho (a)=0$ if, and only if, $\sigma ^{A}(a)=\{0\}$ meanwhile $\rho
_{m}(a)=0 $ if and only if $\sigma _{m}^{A}(a)=\{0\}.$ In the spirit of the
associative case, we say that $A$ is \textbf{spectrally semisimple} if zero
is the unique ideal of $A$ contained in the set $\{a\in A:\rho (a)=0\}.$
Similarly, we say that $A$ is \textbf{multiplicatively semisimple} or 
\textbf{m-semisimple }if \ zero is the unique ideal of $A$ contained in the
set $\{a\in A:\rho _{m}(a)=0\}$ (see in Definition \ref{semisimplicity}).

It is known that \ if the algebra $A$ is semisimple (that is $Rad(A)=\{0\}$)
then $A$ is spectrally semisimple; that if $A$ is spectrally semisimple then 
$A$ is m-semisimple; and also that if $A$ is associative then, these three
notions of semisimplicity are equivalent (see Proposition \ref{rela}). In
contrast to the associative case we provide, in Example \ref{exafinal}, 
\textit{an evolution algebra }$A$\textit{\ which is a radical algebra and
nevertheless }$A$ \textit{is m-semisimple}. This shows how far these three
notions of semisimplicity can be in the non-associative framework.

The notion of m-semisimplicity was used in \cite{Ma-Ve3} to prove the\textit{%
\ automatic continuity of every surjective homomorphism from a Banach
algebra onto a m-semisimple Banach algebra.}

\section{About the existence of a unit in an evolution algebra with
arbitrary dimension}

We begin by showing that an evolution algebra has a unit only in very
special cases. Indeed, an infinite-dimensional evolution algebra has no a
unit, as we prove in the next proposition. Previously, we establish a notion
that will be very helpful in our approach.

\begin{definition}
\emph{Let }$A$\emph{\ be an evolution algebra, and }$B:=\{e_{i}:i\in \Lambda
\}$\emph{\ a natural basis of }$A.$\emph{\ If }$a\in A$\emph{\ is such that }%
$a=\sum_{i\in \Lambda }\alpha _{i}e_{i}$\emph{\ then, the support of }$a$ 
\emph{(respect to }$B$\emph{)} \emph{is defined as} \emph{\ }%
\begin{equation*}
\Lambda _{a}^{B}:=\{i\in \Lambda :\alpha _{i}\neq 0\}.
\end{equation*}

\emph{Similarly, if }$S$\emph{\ is a non-void subset of }$A$\emph{\ then the
support of }$S$ \emph{(respect to }$B$\emph{)} \emph{\ is the set given by }%
\begin{equation*}
\Lambda _{S}^{B}:=\dbigcup\limits_{a\in S}\Lambda _{a}.
\end{equation*}

\emph{If there is no confusion about the prefixed natural basis then we
write simply }$\Lambda _{a}$\emph{\ and }$\Lambda _{S},$\emph{\ respectively}%
$.$
\end{definition}

\begin{proposition}
\label{unidad}Let $A$ be an evolution algebra.

$\mathrm{(i)}$ If $\dim A=\infty $ then, $A$ does not have a unit.

$\mathrm{(ii)}$\ If $\dim A<\infty $ then, $A$ has a unit if, and only if,
for every natural basis $B:=\{e_{1},...,e_{n}\}$ of $A,$ we have that $%
e_{i}^{2}=\omega _{ii}e_{i}$ \ with $\omega _{ii}\neq 0,$ for $i=1,...,n,$
in whose case the unit of $A$ is given by $e=\frac{1}{\omega _{11}}e_{1}+...+%
\frac{1}{\omega _{nn}}e_{n}.$
\end{proposition}

\begin{proof}
Let $A$ be an evolution algebra and $B:=\{e_{i}:i\in \Lambda \}$ a natural
basis of $A.$ The assertion $\mathrm{(i)}$ is clear because if $\Lambda $ is
infinite, if $e$ is a unit of $A,$ and if $i\notin \Lambda _{e}$ then, $%
ee_{i}=0,$ a contradiction that shows that $A$ does not have a unit. To
prove $\mathrm{(ii)}$, suppose that $\Lambda =\{1,...,n\}.$ If $e$ is a unit
for $A$ then, clearly $e=\alpha _{1}e_{1}+...+\alpha _{n}e_{n}$ with $\alpha
_{i}\neq 0$ for every $i=1,...,n.$ Since 
\begin{equation*}
e_{i}^{2}=e_{i}(e_{i}e)=e_{i}(\alpha _{i}e_{i}^{2})=\alpha _{i}\omega
_{ii}e_{i}^{2}
\end{equation*}%
it follows that either $e_{i}^{2}=0$ or $\omega _{ii}\alpha _{i}=1.$ If  $%
e_{i}^{2}=0$ then $ee_{i}=0,~$\ a contradiction. Therefore $\omega
_{ii}\alpha _{i}=1$ and hence $e=\frac{1}{\omega _{11}}e_{1}+...+\frac{1}{%
\omega _{nn}}e_{n},$ so that $e_{i}^{2}=\omega _{ii}e_{i}$ with $\omega
_{ii}\neq 0,$ for every $i=1,...,n.$ The rest is obvious.
\end{proof}

Following \cite{Ti-Vo} we introduce the next definition.

\begin{definition}
\emph{Let }$A$\emph{\ be an evolution algebra and }$B:=\{e_{i}:i\in \Lambda
\}$\emph{\ a natural basis. We say that }$A$\emph{\ is \textbf{non-degenerate%
} if }$e_{i}^{2}\neq 0$ \emph{for any }$i\in \Lambda .$
\end{definition}

It turns out that the above definition does not depend on the prefixed
natural basis $B$, as it was proved in \cite[Corollary 2.19]{Ca-Si-Ve}. This
is because \textit{the evolution algebra }$A$ \textit{is non-degenerate if,
and only if, }$Ann(A)=\{0\},$ where $Ann(A)$ denotes the annihilator of $A$
(see \cite[Proposition 2.18]{Ca-Si-Ve}). Indeed, $Ann(A)=\{0\}$\textit{\ if
and only if the set }$\Lambda _{0}=\{i\in \Lambda :e_{i}^{2}=0\}$\textit{\
is empty. }

\begin{definition}
\emph{Let }$A$\emph{\ be \ an evolution algebra and }$B:=\{e_{i}:i\in
\Lambda \}$\emph{\ a natural basis of }$A.$\emph{\ Let }$\widetilde{B}$\emph{%
\ be another natural basis of }$A.$ \emph{We say that }$B$\emph{\ and }$%
\widetilde{B}$\emph{\ are \textbf{related } if there exists a family of
non-zero constants }$\{k_{i}\}_{i\in \Lambda }$\emph{\ and a bijection }$%
\sigma :\Lambda \rightarrow \Lambda $\emph{\ such that }$\widetilde{B}%
=\{k_{i}e_{\sigma (i)}:i\in \Lambda \}.$
\end{definition}

The next result is easy to prove and nevertheless relevant because\ it gives
us information about when the natural basis of a non-degenerated evolution
algebra is unique (in the meaning that some other natural basis is related
to it).

\begin{proposition}
\label{relate}Let $B:=\{e_{i}:i\in \Lambda \}$ a natural basis of a
non-degenerate evolution algebra $A$. If the set $\{e_{i}^{2}:i\in \Lambda
\} $ is linearly independent then, any two natural basis of $A$ are related.
\end{proposition}

\begin{proof}
Let $\widetilde{B}:=\{u_{i}:i\in \Lambda \}$ be another natural basis of $A.$%
\ Let $i,j\in \Lambda $ with $i\neq j.$ Since $u_{i}u_{j}=0$ we have that $%
\Lambda _{u_{i}}^{B}\cap \Lambda _{u_{j}}^{B}=\emptyset ,$ as the set $%
\{e_{i}^{2}:i\in \Lambda \}$ is linearly independent. Therefore, if we prove
that $card(\Lambda _{u_{i}}^{B})=1$ for every $i\in \Lambda $ then, the
result follows, because\ this means that $u_{i}=\alpha _{j_{0}}e_{j_{0}}$
(with $\alpha _{j_{0}}\neq 0$) for a unique $j_{0}\in \Lambda .$ To the
contrary, suppose that there exists $i_{0}\in \Lambda $ such that $%
card(\Lambda _{u_{i_{0}}}^{B})\neq 1.$ Let $i_{1},i_{2}\in \Lambda
_{u_{i_{0}}}^{B}$ with $i_{1}\neq i_{2}.$ Then, $u_{i_{0}}=\alpha
e_{i_{1}}+\beta e_{i_{2}}+\sum_{k\in \Lambda _{0}}\gamma _{k}e_{k}$ where $%
\alpha \beta \neq 0,$ and $\Lambda _{0}=$ $\Lambda
_{u_{i_{0}}}^{B}\backslash \{i_{1},i_{2}\}.$ Moreover $i_{1},i_{2}\notin
\Lambda _{u_{j}}^{B}$ for every $j\in \Lambda \backslash \{i_{0}\}.$ Denote
by $\pi _{j}$ projection of $A$ on $\mathbb{K}e_{j}.$ If $%
e_{i_{1}}=\sum_{i\in \Lambda _{e_{i_{1}}}^{\widetilde{B}}}\alpha _{i}u_{i}$
then, either $i_{0}\notin \Lambda _{e_{i_{1}}}^{\widetilde{B}}$ in whose
case $\pi _{i_{1}}(e_{i_{1}})=\pi _{i_{1}}(\sum_{i\in \Lambda _{e_{i_{1}}}^{%
\widetilde{B}}}\alpha _{i}u_{i})=0,$ a contradiction as $\pi
_{i_{1}}(u_{j})=0$ if $j\in \Lambda \backslash \{i_{0}\}$, or $i_{0}\in
\Lambda _{e_{i_{1}}}^{\widetilde{B}}$ in whose case $\pi
_{i_{2}}(e_{i_{1}})=\pi _{i_{2}}(\sum_{i\in \Lambda _{e_{i_{1}}}^{\widetilde{%
B}}}\alpha _{i}u_{i})\neq 0,$ another contradiction. This proves that $%
card(\Lambda _{u_{i_{0}}}^{B})=1$ as desired.
\end{proof}

In the above proposition, the hypothesis that the set $\{e_{i}^{2}:i\in
\Lambda \}$ is linearly independent cannot be removed as the next example
shows.

\begin{example}
\emph{Let }$A\,$\emph{be an evolution algebra with natural basis }$%
B:=\{e_{1},e_{2}\}$\emph{\ where }$e_{1}^{2}=e_{1}$\emph{\ and }$%
e_{2}^{2}=-e_{1}.$\emph{\ Then }$A$\emph{\ is non-degenerate and }$%
\widetilde{B}:=\{u_{1},u_{2}\}$\emph{\ with }$u_{1}=e_{1}+e_{2}$\emph{\ and }%
$u_{2}=e_{1}-e_{2}$\emph{\ is a natural basis of }$A$\emph{\ which is not
related to }$B.$
\end{example}

\begin{corollary}
Every two natural basis of a simple evolution algebra are related.
\end{corollary}

\begin{proof}
If $A$ is a simple evolution algebra with a natural basis $B:=\{e_{i}:i\in
\Lambda \}$ then, from \cite[Proposition 4.1]{Ca-Si-Ve}, it follows that $A$
is non-degenerate, that $\{e_{i}^{2}:i\in \Lambda \}$ is a linearly
independent set, and that $A=lin\{e_{i}^{2}:i\in \Lambda \},$ and hence the
above proposition applies.
\end{proof}

The following result is a direct consequence of Proposition \ref{relate}.

\begin{corollary}
If $A$ is a non-degenerate finite dimensional evolution algebra and $B$ is a
natural basis such that $\det M_{B}(A)\neq 0$ then every two natural basis
of $A$ are related.
\end{corollary}

Particular examples of algebras satisfying the hypothesis of the above
result are finite-dimensional simple evolution algebras (see \cite[Corollary
4.10]{Ca-Si-Ve}).

According with \cite[p. 18, Remark 2]{Tian} we introduce the next definition.

\begin{definition}
\emph{Let }$A$\emph{\ be an evolution algebra and }$B:=\{e_{i}:i\in \Lambda
\}$\emph{\ a natural basis of }$A.$\emph{\ We say that }$A$\emph{\ is a%
\textbf{\ non-zero trivial evolution algebra} if }$e_{i}^{2}=\omega
_{ii}e_{i}$\emph{\ with }$\omega _{ii}\neq 0,$\emph{\ for every }$i\in
\Lambda .$ \emph{This means that the structure matrix }$M_{B}(A)$ \emph{is
diagonal and that }$A$\emph{\ is non-degenerate (i.e. }$M_{B}(A)$ \emph{%
diagonal with non-zero entries in the diagonal).}
\end{definition}

\begin{remark}
\label{ulti}\textrm{By definition, a non-zero trivial evolution algebra is
non-degenerate. Therefore, from Proposition \ref{relate}, it follows that
the property of being a non-zero trivial evolution algebra does not depend
on the natural basis considered. }
\end{remark}

Nevertheless, we provide an example of an evolution algebra $A$ (obviously
degenerate)\ having two natural basis $B$ and $\widetilde{B}$ such that $%
M_{B}(A)~$\ is diagonal meanwhile $M_{\widetilde{B}}(A)$ is not (so that $B$
and $\widetilde{B}$ are not related).

\begin{example}
\label{diago}\emph{Let }$A$\emph{\ be the evolution algebra with natural
basis }$B:=\{e_{1},e_{2}\}$\emph{\ and product given by }$e_{1}^{2}=0$\emph{%
\ and }$e_{2}^{2}=e_{2}.$\emph{\ Consider the natural basis }$\widetilde{B}%
:=\{u_{1}.u_{2}\}$ \emph{\ where }$u_{1}=e_{1}$\emph{,}$\ $\emph{\ and }$%
u_{2}=e_{1}+e_{2}.$\emph{\ Then we have that the structure matrix }$%
M_{B}(A)=\left( 
\begin{array}{cc}
0 & 0 \\ 
0 & 1%
\end{array}%
\right) $ \emph{is diagonal meanwhile }$M_{\widetilde{B}}(A)=\left( 
\begin{array}{rr}
0 & -1 \\ 
0 & 1%
\end{array}%
\right) $\emph{\ is not diagonal.}
\end{example}

From Proposition \ref{unidad} we obtain the following result.

\begin{corollary}
\label{anterior}An evolution algebra $A$ has a unit if and only if $\,A$ is
a finite-dimensional non-zero trivial evolution algebra.
\end{corollary}

\begin{remark}
\textrm{In \ \cite[Section 3.1.2, Proposition 1]{Tian} it is established
that "an evolution algebra has a unitary element if and only if it is a
non-zero trivial evolution algebra". Comparing this result with Proposition %
\ref{unidad} we conclude that the finite-dimension of the algebra seems to
be implicitly assumed. }
\end{remark}

\begin{corollary}
\label{unidad2}Let $A$ be an evolution algebra and $B$ a natural basis of $%
A. $ Then the following assertions are equivalent:

$\mathrm{(i)}$ $A$ has a unit.

$\mathrm{(ii)}$\ $A$ is finite-dimensional, $M_{B}(A)$ is diagonal, and has
non-zero entries.

$\mathrm{(iii)}\ A$ finite-dimensional, non-degenerated, and the structure
matrix $M_{B}(A)$ is diagonal.
\end{corollary}

\begin{proof}
$\mathrm{(i)\Longleftrightarrow (ii)}$ is clear from Corollary \ref{anterior}%
, and $\mathrm{(ii)\Longleftrightarrow (iii)}$ follows from the fact if $%
M_{B}(A)$ is diagonal then $A$ is non-degenerated (i.e. every column of $%
M_{B}(A)$ is non-zero) if and only if every entry in the diagonal is
non-zero.
\end{proof}

As usual, if $A$ is an algebra then we define the formal \textbf{unitization}
of $A$ as the algebra $A_{1}:=A\oplus \mathbb{K}\mathbf{1}$ endowed with the
product $(a+\lambda \mathbf{1})(b+\mu \mathbf{1})=ab+\lambda b+\mu a+\lambda
\mu \mathbf{1},$ to obtain an algebra $A_{1}$ with a unit, $\mathbf{1}$,
containing $A$ as an ideal (see \cite{Ve}).

Next, we prove that the unitization $A_{1}:=A\oplus \mathbb{K}\mathbf{1}$ of
an evolution algebra $A$ is an evolution algebra if, and only if, $A$ has a
unit.

\begin{lemma}
\label{Lemauni}If $A_{1}$ is a finite-dimensional evolution algebra and if $%
B_{1}$ is a natural basis of $A_{1},$ then $B_{1}:=B\cup \{\lambda (e-%
\mathbf{1)\}}$ for some $\lambda \neq 0,$ where $B:=\{e_{i}:i\in \Lambda \}$
is a natural basis of $A$ such that $e_{i}^{2}=\omega _{ii}e_{i},$ with $%
\omega _{ii}\neq 0,$ for every $i\in \Lambda ,$ and $e=\sum_{i\in \Lambda }%
\frac{1}{\omega i}e_{i}$ is a unit of $A.$
\end{lemma}

\begin{proof}
Let $A$ be a finite-dimensional algebra such that $A_{1}:=A\oplus \mathbb{K}%
\mathbf{1}$ is an evolution algebra. Let $B_{1}:=\{e_{i}+\lambda _{i}$ $%
\mathbf{1:}$ $i\in \Lambda _{1}\}$ be a natural basis of $A_{1}.$ Since for $%
i\neq j,$%
\begin{equation*}
(e_{i}+\lambda _{i}\mathbf{1)}(e_{j}+\lambda _{j}\mathbf{1)}=0,
\end{equation*}%
we have that there is a unique $i_{0}\in \Lambda _{1}$ such that $\lambda
_{i_{0}}\neq 0.$ Indeed, if $i,j\in \Lambda _{1}$ with $i\neq j$  are such
that $\lambda _{i}\lambda _{j}\neq 0$ then $(e_{i}+\lambda _{i}$ $\mathbf{1)}%
(e_{j}+\lambda _{j}$ $\mathbf{1)}\neq 0,$ a contradiction. Therefore, for $%
\Lambda :=\Lambda _{1}\backslash \{i_{0}\},$ the natural basis $B_{1}$ can
be written as follows,%
\begin{equation*}
B_{1}:=\{e_{i}:i\in \Lambda \}\cup \{e_{i_{0}}+\lambda _{i_{0}}\mathbf{1\},}
\end{equation*}%
where $\lambda _{i_{0}}\neq 0.$ Consequently $A$ is spanned by $\{e_{i}:i\in
\Lambda \},$ and hence $A$ is an evolution algebra and $B:=\{e_{i}:i\in
\Lambda \}$ a natural basis of \ $A.$

On the other hand, since \thinspace $A_{1}$ is an evolution algebra with a
unit $\mathbf{1}$, by Corollary \ref{unidad2}, for every $i\in \Lambda ,$
there exists $\omega _{ii}\neq 0$ such that $e_{i}^{2}=\omega _{ii}e_{i},$
as well as $\omega _{i_{0}i_{0}}\neq 0$ such that $(e_{i_{0}}+\lambda
_{i_{0}}\mathbf{1)}^{2}=\omega _{i_{0}i_{0}}(e_{i_{0}}+\lambda _{i_{0}}%
\mathbf{1).}$ It follows that $e=\sum_{i\in \Lambda }\frac{1}{\omega _{ii}}%
e_{i}$ is a unit for $A.$ Moreover, if $e_{i_{0}}=$ $\sum_{j\in \Lambda
}\beta _{j}e_{j}$ then, since $e_{i}(e_{i_{0}}+\lambda _{i_{0}}\mathbf{1)=}$ 
$0$ we have that 
\begin{equation*}
e_{i}=\frac{-1}{\lambda _{i_{0}}}e_{i}e_{i_{0}}=\frac{-1}{\lambda _{i_{0}}}%
e_{i}\sum_{j\in \Lambda }\beta _{j}e_{j}=\frac{-1}{\lambda _{i_{0}}}\beta
_{i}e_{i}^{2}=\frac{-1}{\lambda _{i_{0}}}\beta _{i}\omega _{ii}e_{i}.
\end{equation*}%
Therefore $\beta _{i}=-\frac{\lambda _{i_{0}}}{\omega _{ii}}$ for every $%
i\in \Lambda ,$ so that 
\begin{equation*}
e_{i_{0}}=\sum_{i\in \Lambda }\beta _{i}e_{i}=-\lambda _{i_{0}}\sum_{j\in
\Lambda }\frac{1}{\omega _{ii}}e_{i}=-\lambda _{i_{0}}e.
\end{equation*}%
From the equality $(e_{i_{0}}+\lambda _{i_{0}}\mathbf{1)}^{2}=\omega
_{i_{0}i_{0}}(e_{i_{0}}+\lambda _{i_{0}}\mathbf{1),}$ we obtain that $\omega
_{i_{0}i_{0}}=\lambda _{i_{0}}.$ Thus \ $B_{1}:=\{e_{i}:i\in \Lambda \}\cup
\{-\lambda _{i_{0}}(e-\mathbf{1)\},}$ where $\lambda _{i_{0}}\neq 0$ and $%
B:=\{e_{i}:i\in \Lambda \}$ is a natural basis of $A$ such that $%
e_{i}^{2}=\omega _{ii}e_{i},$ with $\omega _{ii}\neq 0,$ for every $i\in
\Lambda ,$ and $e=\sum_{i\in \Lambda }\frac{1}{\omega _{ii}}e_{i}$ is a unit
of $A,$ as desired$.$
\end{proof}

\begin{proposition}
Let $A$ be an algebra and let $A_{1}$ its unitization. Then the following
assertions are equivalent:

$\mathrm{(i)}$ $A_{1}$ is an evolution algebra.

$\mathrm{(ii)}$ $A$ is an evolution algebra with a unit.

$\mathrm{(iii)}$ $A$ is a finite-dimensional non-zero trivial evolution
algebra.

$\mathrm{(iv)}$ $A_{1}$ is a finite-dimensional non-zero trivial evolution
algebra.
\end{proposition}

\begin{proof}
$\mathrm{(i)\Rightarrow (ii).}$ From Corollary \ref{anterior} we have that $%
\dim A_{1}<\infty $ and applying Lemma \ref{Lemauni} we obtain $\mathrm{(ii).%
}$ That $\mathrm{(ii)\Longleftrightarrow (iii)}$ and that $\mathrm{%
(i)\Longleftrightarrow (iv)}$ are obvious by Corollary \ref{anterior} \ To
prove $\mathrm{(iii)\Rightarrow (iv),}$ let $A$ be a\ non-zero trivial
evolution algebra and $B:=\{e_{i}:i\in \Lambda \}$ a natural basis. Then,
for every $i\in \Lambda ,$ let $\omega _{ii}\neq 0$ be such that $%
e_{i}^{2}=\omega _{ii}e_{i}.$ It follows that $e=\sum_{i\in \Lambda }\frac{1%
}{\omega _{i}}e_{i}$ is a unit of $A,$ and $B\cup \{\mathbf{1-}e\mathbf{\}}$
is a natural basis for $A_{1}$ whose structure matrix is diagonal with
non-zero entries, as $\mathbf{1-}e$ is idempotent.
\end{proof}

We conclude that except in very special cases, the unitization process is
incompatible with the property of being an evolution algebra.

\begin{corollary}
\label{unifinal}Every non-trivial evolution algebra does not have a unit,
and its unitization is not an evolution algebra.
\end{corollary}

\section{The Jacobson radical of  an evolution algebra}

The Jacobson radical of an associative algebra is the intersection of its
all primitive ideals (see for instance \cite{Dales}). This notion was
generalized to arbitrary non-associative algebras in \cite{Ma-Ve2} by
defining a \textbf{primitive ideal} as the biggest two sided ideal contained
in a maximal one-sided modular ideal. Since evolution algebras are
commutative, we avoid the adjectives "left" and "right" and we talk only
about "modular ideals" and "modular units".

We recall that an ideal $M$ of a commutative algebra $A$ is said to be a%
\textbf{\ modular ideal }$\ $if there exists $u\in A$ such that $a-au\in M,$
for every $a\in A$ (in other words $A(1-u)\subseteq M$). In this case, we
say that $u$ is a\textbf{\ modular unit} for $M.$ Therefore, according with 
\cite{Ma-Ve2}, we have the following definition.

\begin{definition}
\emph{Let }$A$\emph{\ be an evolution algebra. }$\emph{A}$ \emph{\textbf{%
primitive} ideal of }$A$\emph{\ is a maximal modular ideal, and the Jacobson
radical of }$A$\emph{\ is the intersection of all the primitive ideals of }$%
A.$
\end{definition}

Our next goal is to characterize the primitive ideals of an evolution
algebra.

\begin{theorem}
\label{cuadrado}Let $A$ be an evolution algebra, and $B:=\{e_{i}:i\in
\Lambda \}$ a natural basis of $A.$ Let $M$ be a modular ideal with support $%
\Lambda _{M}.$ Then:

$\mathrm{(i)}$ $e_{i}\in M$ if and only if $e_{i}^{2}\in M.$

$\mathrm{(ii)}$ $M=lin\{e_{i}:i\in \Lambda _{M}\}$. Consequently, $M$ is
proper if and only if $\Lambda _{M}\neq \Lambda .$

$\mathrm{(iii)}$ $\Lambda \backslash \Lambda _{M}$ is finite. Moreover, if $%
u $ is a modular unit for $M$ then $\Lambda _{M}\cup \Lambda _{u}=\Lambda .$

$\mathrm{(iv)}$ If $M$ is proper then $\omega _{ii}\neq 0,$ for every $i\in
\Lambda \backslash \Lambda _{M},$ and 
\begin{equation*}
u_{0}=\sum\limits_{i\in \Lambda \backslash \Lambda _{M}}\frac{1}{\omega _{ii}%
}e_{i}
\end{equation*}%
is a modular unit for $M$. In fact, $u\in A$ is a modular unit for $M$ if,
and only if, $u=m+u_{0}$ for some $m\in M.$
\end{theorem}

\begin{proof}
$\mathrm{(i)}$\ Obviously if $e_{i}\in M$ then $e_{i}^{2}\in M.$ Conversely,
suppose that $e_{i}^{2}\in M$ and let $u$ be a modular unit for $M$. Let $%
\pi _{i}(u)=\alpha _{i},$ where, as usual, $\pi _{i}$ denotes the projection
of $A$ on $\mathbb{K}e_{i}$. Since $e_{i}(\mathbf{1}-u)=e_{i}-\alpha
_{i}e_{i}^{2}\in M$ and $e_{i}^{2}\in M$ we have that\ $e_{i}\in M,$ as
desired.

$\mathrm{(ii)}$\ If $\pi _{i}(M)\neq \{0\}$ then $e_{i}^{2}\in M$ and, by $%
\mathrm{(i),}$\ we have that $e_{i}\in M$ so that 
\begin{equation*}
lin\{e_{i}:i\in \Lambda _{M}\}\subseteq M.
\end{equation*}%
Conversely, if $a=\sum\limits_{i\in \Lambda _{a}}\beta _{i}e_{i}\in M$ (with 
$\beta _{i}\neq 0$ for every \thinspace $i\in \Lambda _{a}$) then $\Lambda
_{a}\subseteq \Lambda _{M},$ so that $a\in lin\{e_{i}:i\in \Lambda _{M}\}.$
Thus $M=lin\{e_{i}:i\in \Lambda _{M}\},$ and the rest is clear$.$

$\mathrm{(iii)}$ If $u$ is a modular unit for $M$ and if $i\in (\Lambda
\backslash \Lambda _{M})\cap (\Lambda \backslash \Lambda _{u})$ then $%
e_{i}-e_{i}u=e_{i}\,$\ is in $M,$ a contradiction. It follows that $(\Lambda
\backslash \Lambda _{M})\subseteq \Lambda _{u}$ so that $\Lambda \backslash
\Lambda _{M}$ is finite. Moreover 
\begin{equation*}
\Lambda \backslash (\Lambda _{M}\cup \Lambda _{u})=(\Lambda \backslash
\Lambda _{M})\cap (\Lambda \backslash \Lambda _{u})=\emptyset ,
\end{equation*}%
that is $\Lambda _{M}\cup \Lambda _{u}=\Lambda .$

$\mathrm{(iv)}$ Let $u=\sum\limits_{i\in \Lambda _{u}}\alpha _{i}e_{i}$ \ be
a modular unit for $M.$ If $\Lambda _{u}\cap \Lambda _{M}=\emptyset $ \
then, take $m=0$ and, otherwise, let $m=\sum\limits_{i\in \Lambda _{u}\cap
\Lambda _{M}}\alpha _{i}e_{i}.$ Since $(\Lambda \backslash \Lambda
_{M})\subseteq \Lambda _{u}$ by $\mathrm{(iii),}$we obtain that 
\begin{equation*}
u=\sum\limits_{i\in \Lambda \backslash \Lambda _{M}}\alpha _{i}e_{i}+m,
\end{equation*}%
for some $m\in M.$ For $i\in \Lambda \backslash \Lambda _{M}$ we have that 
\begin{equation*}
e_{i}(e_{i}-e_{i}u)=(1-\alpha _{i}\omega _{ii})e_{i}^{2}\in M,
\end{equation*}%
so that $1-\alpha _{i}\omega _{ii}=0$ and hence $\omega _{ii}\neq 0$ and $%
\alpha _{i}=\frac{1}{\omega _{ii}}.$ Therefore, $u=m+u_{0}$ and consequently 
$u_{0}$ is a modular unit for $M$, as desired. This concludes the proof
because, obviously, $u_{0}+m$ is a modular unit of $M,$ for every $m\in M.$
\end{proof}

\bigskip

The set of descendents of a set of indexes\ is a concept introduced in \cite[%
Definitions 3.1]{Ca-Si-Ve} that will play an essential role in describing
the dynamic nature of an evolution algebra$.$

\begin{definition}
\label{descen}\emph{Let }$A$\emph{\ be an evolution algebra, and }$%
B:=\{e_{i}:i\in \Lambda \}$\emph{\ a natural basis of }$A.$ \emph{If
\thinspace }$e_{i}^{2}=\dsum\limits_{j\in \Lambda }\omega _{ji}e_{j}$\emph{\
then, the set of\ \textbf{first-generation descendents} of the index }$\
i\in \Lambda $\emph{\ is defined as }%
\begin{equation*}
D^{1}(i):=\Lambda _{e_{i}^{2}}=\{j\in \Lambda :\omega _{ji}\neq 0\}.
\end{equation*}%
\emph{Similarly, the set of \textbf{second-generation descendents} of }$\
i\in \Lambda $\emph{\ is given by }%
\begin{equation*}
D^{2}(i)=\dbigcup\limits_{j\in \Lambda _{e_{i}^{2}}}\Lambda
_{e_{j}^{2}}=\dbigcup\limits_{j\in D^{1}(i)}D^{1}(j).
\end{equation*}%
\emph{Similarly, the set of\ \textbf{nth-generation descendents} of }$\ i\in
\Lambda $\ \emph{is given by }%
\begin{equation*}
D^{n}(i)=\dbigcup\limits_{j\in D^{n-1}(i)}D^{1}(j).
\end{equation*}%
\emph{Finally the set of\ \textbf{descendents of the index} }$i$\emph{\ is
defined as \thinspace }%
\begin{equation*}
D(i)=\dbigcup\limits_{n\in \mathbb{N}}D^{n}(i).
\end{equation*}%
\emph{On the other hand, the set of descendents of a non-empty set }$\Lambda
_{0}\subseteq \Lambda $\emph{\ \ is defined as}%
\begin{equation*}
D(\Lambda _{0})=\dbigcup\limits_{i\in \Lambda _{0}}D(i).
\end{equation*}
\end{definition}

In \cite[Proposition 3.4]{Ca-Si-Ve} the set of descendents of an index in a
natural basis was characterized as follows.

\begin{proposition}
Let $B:=\{e_{i}\ :\ i\in \Lambda \}$ be a natural basis of an evolution
algebra $A,$ and let $M_{B}(A)=(\omega _{ij}).$ Consider $i_{0},j\in \Lambda 
$ and $m\geq 2$.

$\mathrm{(i)}$ $\ j\in D^{1}(i_{0})$ if and only if $\omega _{ji_{0}}\neq 0,$
in which case $e_{j}e_{i_{0}}^{2}=\omega _{ji_{0}}e_{j}^{2}.$

$\mathrm{(ii)}$ $\ j\in D^{m}(i_{0})$ if, and only if, there exist $%
k_{1},k_{2},\cdots ,k_{m-1}$ such that 
\begin{equation*}
\omega _{jk_{m-1}}\omega _{k_{m-1}k_{m-2}\cdots }\omega _{k_{2}k_{1}}\omega
_{k_{1}i_{0}}\neq 0,
\end{equation*}%
in which case $e_{j}^{2}=(\omega _{jk_{m-1}}\omega _{k_{m-1}k_{m-2}\cdots
}\omega _{k_{2}k_{1}}\omega
_{k_{1}i_{0}})^{-1}e_{j}e_{k_{m-1}}e_{k_{m-2}}\cdots
e_{k_{2}}e_{k_{1}}e_{i_{0}}^{2}.$
\end{proposition}

\begin{corollary}
\label{new}Let $B:=\{e_{i}\ :\ i\in \Lambda \}$ be a natural basis of an
evolution algebra $A.$ Let $m,n\in \mathbb{N}$ be such that $m>n,$ and let $%
i_{0}\in \Lambda .$ Then, $k\in D^{m}(i_{0})$ if, and only if, $k\in D^{n}(j)
$ for some $j\in D^{m-n}(i_{0}).$ Consequently 
\begin{equation*}
D^{m}(i_{0})=\dbigcup\limits_{j\in D^{m-n}(i_{0})}D^{n}(j).
\end{equation*}
\end{corollary}

\begin{proof}
The result follows from the above result and the fact that there exists $%
k_{1},k_{2},\cdots ,k_{m-1}$ such that $\omega _{kk_{m-1}}\omega
_{k_{m-1}k_{m-2}\cdots }\omega _{k_{2}k_{1}}\omega _{k_{1}i_{0}}\neq 0,$ if
and only if there exists $k_{1},k_{2},\cdots ,k_{n-1}$ such that 
\begin{equation*}
\omega _{kk_{n-1}}\omega _{k_{n-1}k_{n-2}\cdots }\omega _{k_{2}k_{1}}\omega
_{k_{1}j}\neq 0,
\end{equation*}%
and $\widetilde{k}_{1},\widetilde{k}_{2}\cdots ,\widetilde{k}_{m-n-1}$ such
that $\omega _{j\widetilde{k}_{m-n-1}}\omega _{\widetilde{k}_{m-n-1}%
\widetilde{k}_{m-n-2}\cdots }\omega _{\widetilde{k}_{2}\widetilde{k}%
_{1}}\omega _{\widetilde{k}_{1}i_{0}}\neq 0.$
\end{proof}

\begin{proposition}
Let $A$ be an evolution algebra, and $B:=\{e_{i}:i\in \Lambda \}$ a natural
basis of $A.$ If $M$ is a non-zero modular ideal with support $\Lambda _{M}$
then, $D(i)\backslash \{i\}\subseteq \Lambda _{M}$, for every $i\in \Lambda
\backslash \Lambda _{M}.$
\end{proposition}

\begin{proof}
If $M=A$ then $\Lambda _{M}=\Lambda $ and there is nothing to prove.
Otherwise, by Theorem \ref{cuadrado}, we have that $\omega _{ii}\neq 0,$ for
every $i\in \Lambda \backslash \Lambda _{M},$ and $u_{0}=\sum\limits_{i\in
\Lambda \backslash \Lambda _{M}}\frac{1}{\omega _{ii}}e_{i}$ is a modular
unit for $M$. Let $i\in \Lambda \backslash \Lambda _{M}$. For every $j\in
D^{1}(i)\backslash \{i\}=\Lambda _{e_{i}^{2}}\backslash \{i\}$ we have $\pi
_{j}(e_{i}-e_{i}u_{0})=\omega _{ji}\neq 0,$ with $\ e_{i}-e_{i}u_{0}\in M,$
so that $j\in \Lambda _{M}.\,\ $Therefore $D^{1}(i)\backslash \{i\}\subseteq
\Lambda _{M}$. On the other hand, 
\begin{equation*}
D^{2}(i)\backslash \{i\}=\dbigcup\limits_{j\in D^{1}(i)}D^{1}(j)\backslash
\{i\},
\end{equation*}%
and if $k\in D^{1}(j)\backslash \{i\}$ for some $j\in D^{1}(i)$ then, $%
\omega _{kj}\neq 0$ and $e_{k}e_{j}^{2}=\omega _{kj}e_{k}^{2}$. If $j=i$
then $k\in D^{1}(i)\backslash \{i\}\subseteq \Lambda _{M}.$ Otherwise, $j\in
D^{1}(i)\backslash \{i\}\subseteq \Lambda _{M}$ and then $e_{k}^{2}=\frac{1}{%
\omega _{kj}}e_{k}e_{j}^{2}\in M$ because $e_{j}^{2}\in M.$ Thus $e_{k}\in
M, $ by Theorem \ref{cuadrado}$\mathrm{,}$ and hence $k\in \Lambda _{M}$.
This proves that $D^{2}(i)\backslash \{i\}\subseteq \Lambda _{M}$ and, by an
inductive process, we obtain \ that $D^{n}(i)\backslash \{i\}\subseteq
\Lambda _{M}$ for every $n\in \mathbb{N},$ or equivalently $D(i)\backslash
\{i\}\subseteq \Lambda _{M}$.
\end{proof}

Next we characterize the modular ideals of an evolution algebra by means of
the descendents set of a set of \ indexes.

\begin{corollary}
\label{coro}Let $A$ be an evolution algebra and $B:=\{e_{i}:i\in \Lambda \}$
a natural basis. A non-zero proper subset $M\subseteq A$ is a modular ideal
of $A$ if, and only if, $M=lin\{e_{j}:j\in \Lambda _{0}\}$ for some
non-empty proper subset $\Lambda _{0}\subseteq \Lambda $ satisfying the
following conditions:

$\mathrm{(i)}$\ $\Lambda \backslash \Lambda _{0}$ is finite and $\Lambda
\backslash \Lambda _{0}\subseteq \Lambda \backslash D(\Lambda _{0}),$

$\mathrm{(ii)}$\ $i\in D(i)$ and $D(i)\backslash \{i\}\subseteq \Lambda
_{0}, $ for every $i\in \Lambda \backslash \Lambda _{0};$

in which case $u=\sum\limits_{i\in \Lambda \backslash \Lambda _{0}}\frac{1}{%
\omega _{ii}}e_{i}$ is a modular unit for $M$.
\end{corollary}

\begin{proof}
Let $\Lambda _{0}\subseteq \Lambda $ satisfying $\mathrm{(i)},$ and $\mathrm{%
(ii)}$. Then, the condition $\mathrm{(i)}$ assures that $M:=lin\{e_{i}:i\in
\Lambda _{0}\}$ is an ideal as $D(\Lambda _{0})\subseteq \Lambda _{0},$
which is proper as $\Lambda \backslash \Lambda _{0}$ is finite. For every $%
i\in \Lambda \backslash \Lambda _{0},$ by condition $\mathrm{(ii)}$ we have
that $\omega _{ii}\neq 0$ (as $i\in D(i)$) so that $u=\sum\limits_{i\in
\Lambda \backslash \Lambda _{0}}\frac{1}{\omega _{ii}}e_{i}$ is well
defined. Since $D(i)\backslash \{i\}\subseteq \Lambda _{0},$ we have \ that 
\begin{equation*}
e_{i}-e_{i}u\in lin\{e_{j}:j\in D(i)\backslash \{i\}\}\subseteq
lin\{e_{j}:j\in \Lambda _{0}\}=M,
\end{equation*}%
for every $i\in \Lambda ,$ which implies that $M$ is a (proper)\ modular
ideal and $u=\sum\limits_{i\in \Lambda _{0}}\frac{1}{\omega _{ii}}e_{i}$ a
modular unit for $M.$ Conversely, by Theorem \ref{cuadrado} and the above
proposition, we conclude that every proper\ modular ideal of $A$ is of this
type (take $\Lambda _{0}=\Lambda _{M}$).
\end{proof}

From the above result we characterize the maximal modular ideals of an
evolution algebra, and therefore its Jacobson radical.

\begin{corollary}
\label{coroa}Let $A$ be an evolution algebra, $B:=\{e_{i}:i\in \Lambda \}$ a
natural basis, and $M_{B}(A)=(\omega _{ij})$ the structure matrix. Then $M$
is a maximal modular ideal of $A$ if, and only if, $M=lin\{e_{i}:i\in
\Lambda \backslash \{i_{0}\}\}$ for some $i_{0}\in \Lambda $ satisfying that 
$\omega _{i_{0}i_{0}}\neq 0$, $\ $and that $i_{0}\notin D(k)$ $\ $for every $%
k\in \Lambda \backslash \{i_{0}\},$ in which case $\frac{1}{\omega
_{i_{0}i_{0}}}e_{i_{0}}$ is a modular unit for $M.$ Consequently every
maximal modular ideal of $A$ has codimension one.
\end{corollary}

\begin{proof}
Suppose that $i_{0}$ is such that $\omega _{i_{0}i_{0}}\neq 0$ and $%
i_{0}\notin D(k)$ for every $k\neq i_{0}.$ Then $M=lin\{e_{i}:i\neq i_{0}\}$
is a maximal modular ideal as it can be deduced from Corollary \ref{coro}.
Indeed, for $\Lambda _{0}=\Lambda \backslash \{i_{0}\}$ we have that

$\mathrm{(i)}$ $\Lambda \backslash \Lambda _{0}=\{i_{0}\}~$\ is finite, and $%
\{i_{0}\}=\Lambda \backslash \Lambda _{0}\subseteq \Lambda \backslash
D(\Lambda _{0})$ because $i_{0}\notin D(k)$ for $k\neq i_{0},$ so that $%
i_{0}\notin D(\Lambda _{0})=D(\Lambda \backslash \{i_{0}\}).$

$\mathrm{(ii)}$ $\ i\in D(i)$ (that is $\omega _{i_{0}i_{0}}\neq 0$) and,
obviously, $D(i_{0})\backslash \{i_{0}\}\subseteq \Lambda _{0}=\Lambda
\backslash \{i_{0}\}$.

Therefore by the above corollary we obtain that $M$ is a modular ideal of
codimension one, and hence maximal.

Conversely, let $M$ be a maximal modular ideal. By Corollary \ref{coro} we
have that $M=lin\{e_{k}:k\in \Lambda _{M}\}$ with

$\mathrm{(i)}$ $\Lambda \backslash \Lambda _{M}$ \ finite, and $\Lambda
\backslash \Lambda _{M}\subseteq \Lambda \backslash D(\Lambda _{M})$,

$\mathrm{(ii)}$ $\omega _{ii}\neq 0$ and $D(i)\backslash \{i\}\subseteq
\Lambda _{M}$ for every $i\in \Lambda \backslash \Lambda _{M}.$

Suppose that $i,j\in \Lambda \backslash \Lambda _{M}$ with $i\neq j.$ Then $%
i\notin D(j)$ and $j\notin D(i)$ by the hypothesis $\mathrm{(ii).}$ Indeed,
if $j\in \Lambda \backslash \Lambda _{M}$ and $D(i)\backslash \{i\}\subseteq
\Lambda _{M}$ then $j\notin D(i)\backslash \{i\}$ and hence $j\notin D(i).$
Similarly $i\notin D(j)$. Note that $\Lambda _{0}:=\Lambda _{M}\cup \{i\}$
is a proper subset of $\Lambda $ such that \ 

$\mathrm{(i)}$ $\Lambda \backslash \Lambda _{0}$ is finite and $\Lambda
\backslash \Lambda _{0}\subseteq \Lambda \backslash D(\Lambda _{0})$ since
from the inclusion $D(i)\backslash \{i\}\subseteq \Lambda _{M}$ (by keeping
in mind that $D(\Lambda _{M})\subseteq \Lambda _{M}$) it follows that%
\begin{equation*}
D(\Lambda _{0})=D(\Lambda _{M}\cup \{i\})\subseteq \Lambda _{M}\cup D(i)\cup
\{i\}\subseteq \Lambda _{M}\cup \{i\}=\Lambda _{0},
\end{equation*}

$\mathrm{(ii)}$ $\omega _{kk}\neq 0$, for every $k\in \Lambda \backslash
\Lambda _{0}$ (because $\omega _{kk}\neq 0$ for $k\in \Lambda \backslash
\Lambda _{M}$ and $\Lambda \backslash \Lambda _{0}\subseteq $ $\Lambda
\backslash \Lambda _{M}$). Also $D(k)\backslash \{k\}\subseteq \Lambda _{0}$
for every $k\in \Lambda \backslash \Lambda _{0},$ (because $D(k)\backslash
\{k\}\subseteq \Lambda _{M}\subseteq \Lambda _{0}$). Consequently, by
Corollary \ref{coro}, we have that $lin\{e_{k}:k\in \Lambda _{M}\cup \{i\}\}$
is a proper modular ideal containing $M.$ This contradiction proves that $%
\Lambda \backslash \Lambda _{M}=\{i_{0}\}$ for some $i_{0}\in \Lambda $ such
that $\omega _{i_{0}i_{0}}\neq 0$ \ and $D(i_{0})\backslash
\{i_{0}\}\subseteq \Lambda _{M}$ with $D(\Lambda _{M})\subseteq \Lambda _{M}$
(that is $D(\Lambda \backslash \{i_{0}\})\subseteq \Lambda \backslash
\{i_{0}\}$). Therefore $\ i_{0}\notin D(k)$ for every $k\in \Lambda
\backslash \{i_{0}\},$ as desired$.$ It also follows that $\frac{1}{\omega
_{i_{0}i_{0}}}e_{i_{0}}$ is a modular unit for $M$. Thus, every maximal
modular ideal of $A$ has codimension one.
\end{proof}

The above result is very helpful in determining the maximal modular ideals
of an evolution algebra and hence its Jacobson radical, as the following
example shows.

\begin{example}
\label{ejemrado}\emph{Let }$A$\emph{\ be the evolution algebra with natural
basis }%
\begin{equation*}
B:=\{e_{1},e_{2},e_{3},e_{4},e_{5},e_{6},e_{7}\}
\end{equation*}%
\emph{and structure matrix given by}%
\begin{equation*}
\left( 
\begin{array}{ccccccc}
\omega _{11} & \omega _{12} & \omega _{13} & 0 & 0 & 0 & \omega _{17} \\ 
\omega _{21} & \omega _{22} & \omega _{23} & 0 & 0 & 0 & \omega _{27} \\ 
\omega _{31} & \omega _{32} & \omega _{33} & 0 & 0 & 0 & \omega _{37} \\ 
0 & 0 & 0 & \omega _{44} & 0 & \omega _{46} & 0 \\ 
0 & 0 & 0 & 0 & \omega _{55} & \omega _{56} & 0 \\ 
0 & 0 & 0 & 0 & 0 & \omega _{66} & 0 \\ 
0 & 0 & 0 & 0 & 0 & 0 & \omega _{77}%
\end{array}%
\right) 
\end{equation*}%
\emph{(where }$\omega _{ij}\neq 0$\emph{).~Then, according to Corollary \ref%
{coroa}, the maximal modular ideals of }$A$\emph{\ are the following ones }%
\begin{eqnarray*}
M_{6} &=&lin\{e_{1},e_{2},e_{3},e_{4},e_{5},e_{7}\}, \\
M_{7} &=&lin\{e_{1},e_{2},e_{3},e_{4},e_{5},e_{6}\}\emph{.}
\end{eqnarray*}%
\emph{\ Indeed if }$M_{i}=lin\{e_{j}:j\in \{1,2,3,4,5,6,7\}\backslash
\{i\}\}\,\ $\emph{then, }$M_{1},M_{2},M_{3},M_{4}$\emph{, and }$M_{5}$\emph{%
\ are not even modular ideals. Consequently, }%
\begin{equation*}
Rad(A)=M_{6}\cap M_{7}=lin\{e_{1},e_{2},e_{3},e_{4},e_{5}\}.
\end{equation*}
\end{example}

To characterize easily the Jacobson radical of an evolution algebra we
introduce the following notion.

\begin{definition}
\label{modularindex}\emph{Let }$A$ \emph{be a natural evolution algebra, }$%
B:=\{e_{i}:i\in \Lambda \}$\emph{\ a natural basis of }$A,$\emph{\ and }$%
M_{B}(A)=(\omega _{ij})_{\Lambda \times \Lambda }$\emph{\ the corresponding
structure matrix. We say that }$i_{0}\in \Lambda $\emph{\ is a \textbf{%
modular index} if }$\omega _{i_{0}i_{0}}\neq 0$\emph{\ and }$\omega
_{ji_{0}}=0$\emph{\ if }$i_{0}\neq j.$
\end{definition}

That $i_{0}\in \Lambda $ is a modular index means that the corresponding $%
i_{0}-$file in the structure matrix $M_{B}(A)=(\omega _{ij})_{\Lambda \times
\Lambda }$ consists of zeros except for the entry $\omega _{i_{0}i_{0}}\neq
0.$

In the next result we characterize the modular indexes in terms of the sets
of descendents.

\begin{lemma}
Let $A$ be an evolution algebra, $B:=\{e_{i}:i\in \Lambda \}$ a natural
basis, and $M_{B}(A)=(\omega _{ij})$ the structure matrix. Let $i_{0}\in
\Lambda .$ Then, the following assertions are equivalent:

$\mathrm{(i)}$ $i_{0}$ is a modular index,

$\mathrm{(ii)}$ $\omega _{i_{0}i_{0}}\neq 0$ and $i_{0}\notin D^{1}(j)$, for
every $j\in \Lambda \backslash \{i_{0}\},$

$\mathrm{(iii)}$\ $\omega _{i_{0}i_{0}}\neq 0\ $and $i_{0}\notin D(j),$ $\ $%
for every $j\in \Lambda \backslash \{i_{0}\}.$
\end{lemma}

\begin{proof}
$\mathrm{(i)\Leftrightarrow (ii)}$ is nothing but the definition of a
modular index. To prove $\mathrm{(ii)\Rightarrow (ii)}$ suppose that $%
i_{0}\notin D^{1}(j)$, for every $j\in \Lambda \backslash \{i_{0}\}.$ Then
it follows that $i_{0}\notin D^{2}(j)$ for every $j\in \Lambda \backslash
\{i_{0}\},$ because $D^{2}(j)=\dbigcup\limits_{k\in D^{1}(j)}D^{1}(k),$ and
if $k\in D^{1}(j)$ for $j\in \Lambda \backslash \{i_{0}\}$ and $i_{0}\in
D^{1}(k)$ then, by $\mathrm{(ii),}$ we have that $k=i_{0}$ so that $i_{0}\in
D^{1}(j)$ for $j\in \Lambda \backslash \{i_{0}\},$ a contradiction. Thus $%
i_{0}\notin D^{2}(j)$ for every $j\in \Lambda \backslash \{i_{0}\}.$ Assume
that $i_{0}\notin D^{n}(j)$ for every $j\in \Lambda \backslash \{i_{0}\}.$
Since $D^{n+1}(j)=\dbigcup\limits_{k\in D^{n}(j)}D^{1}(k),$ we have that $%
i_{0}\notin D^{n+1}(j)$ for each $j\in \Lambda \backslash \{i_{0}\},$
because otherwise $i_{0}\in D^{1}(k)$ for $k\in D^{n}(j),$ and by $\mathrm{%
(ii)}$ we have that $k=i_{0}$ so that $i_{0}\in D^{n}(j)$ $\ $for $j\in
\Lambda \backslash \{i_{0}\},$ which \ contradicts the inductive hypothesis.
It follows that $i_{0}\notin D(j)$ $\ $for every $j\in \Lambda \backslash
\{i_{0}\}$ which proves $\mathrm{(iii)}$. Finally $\mathrm{(iii)\Rightarrow
(ii)}$ is obvious because $D^{1}(j)\subseteq D(j),$ for every $j\in \Lambda
. $
\end{proof}

Maximal modular ideals can be characterized in terms of a modular index we
do next.

\begin{corollary}
\label{moduin}Let $A$ be an evolution algebra and $B:=\{e_{i}:i\in \Lambda \}
$ a natural basis. Then $M$ is a maximal modular ideal of $A$ if, and only
if, $M=lin\{e_{i}:i\in \Lambda \backslash \{i_{0}\}\}$ for some modular
index $i_{0}\in \Lambda ,$ in which case $u_{0}:=\frac{1}{\omega
_{i_{0}i_{0}}}e_{i_{0}}$ is a modular unit for $M.$
\end{corollary}

\begin{proof}
The proof follows directly from Corollary \ref{coroa} \ and the above Lemma.
\end{proof}

Since the Jacobson radical,$\ Rad(A),$ of an evolution algebra $A$ is given
by the intersection of all maximal modular ideals of $A,$ the above result
allows us to characterize it very easily.

\begin{corollary}
\label{Jradical}Let $A$ be a non-zero evolution algebra, and $%
B:=\{e_{i}:i\in \Lambda \}$ a natural basis of $A.$ Let $\Lambda
_{m}:=\{i\in \Lambda :$ $i$ \textit{is }modular$\}.$ Then, $A$ is a radical
algebra if, and only if, $\Lambda _{m}=\emptyset $ and $A$ is semisimple if,
and only if, $\Lambda =\Lambda _{m}$. Moreover, $\ $%
\begin{equation*}
Rad(A)=lin\{e_{i}:i\in \Lambda \backslash \Lambda _{m}\}.
\end{equation*}
\end{corollary}

\begin{proof}
By Corollary \ref{moduin} the set of maximal modular ideals of $A$ can be
described as $\{M_{i_{0}}:i_{0}\in \Lambda _{m}\}$ where $%
M_{i_{0}}=lin\{e_{i}:i\in \Lambda \backslash \{i_{0}\}\}.$ Therefore $A$ has
no modular ideals if, and only if, $\Lambda _{m}=\emptyset ,$ in whose case $%
A$ is a radical algebra. On the other hand, if $\Lambda _{m}\neq \emptyset $
then, $Rad(A)=\{0\}$ if, and only, if $\Lambda _{m}=\Lambda ,$ because the
unique maximal modular ideal that does not contain $e_{i_{0}}$ is $M_{i_{0}}$
provided that $i_{0}$ is a modular index (this means that $e_{i_{0}}\in
Rad(A)$ whenever $i_{0}\notin \Lambda _{m}$). Finally, if $\Lambda _{m}$ is
a non-void proper subset of $\Lambda $ then,%
\begin{equation*}
Rad(A)=\dbigcap\limits_{i_{0}\in \Lambda _{m}}M_{i_{0}}=lin\{e_{i}:i\in
\Lambda \backslash \Lambda _{m}\}.
\end{equation*}
\end{proof}

Since modular indexes are very easy to detect ($i_{0}$ is modular if the
corresponding $i_{0}-$file in the structure matrix $M_{B}(A)=(\omega
_{ij})_{\Lambda \times \Lambda }$ consists of zeros except for the entry $%
\omega _{i_{0}i_{0}}\neq 0$), \ thanks to the above result, we obtain that
the Jacobson radical of an evolution algebra is very easy to determine (see
Example \ref{ejemrado}).

\begin{corollary}
\label{nuevo}If $A$ is a non-zero evolution algebra then,

$\mathrm{(i)}$ $A$ is semisimple if, and only if, $A$ a non-zero trivial
evolution algebra,

$\mathrm{(ii)}$ $A/Rad(A)$ is either $\{0\}$ or a non-zero trivial evolution
algebra.
\end{corollary}

\begin{proof}
$\mathrm{(i)}$ Let $B:=\{e_{i}:i\in \Lambda \}$ a natural basis of $A.$ From
the above result $A$ is semisimple if, and only if, $\Lambda =\Lambda _{m},$
which means that all the indexes in $\Lambda $ are modular and hence $%
M_{B}(A)$ is a diagonal matrix with non-zero entries in the diagonal or,
equivalently, that $A$ is a non-zero trivial evolution algebra.

$\mathrm{(ii)}$ Since $A/Rad(A)$ is an evolution algebra by \cite[Lemma 2.10]%
{Ca-Si-Ve}, which is semisimple by \cite[Proposition 7]{Ma-Ve2}, the
assertion $\mathrm{(ii)}$ follows from $\mathrm{(i)}$.
\end{proof}

We have shown that evolution algebras are not semisimple (unless they are
trivial). Moreover for finite-dimensional evolution algebras we have the
following result which is obtained straightforwardly from Corollary \ref%
{anterior}, and Corollary \ref{nuevo}.

\begin{corollary}
\label{coroc}Let $A$ be a finite-dimensional evolution algebra. Then, $A$ is
semisimple if and only if $A$ has a unit.
\end{corollary}

\begin{corollary}
Let $A$ be an evolution algebra and $B:=\{e_{i}:i\in \Lambda \}$ a natural
basis of $A.$ Then $e_{i}\in Rad(A)$ if and only if $e_{i}^{2}\in Rad(A).$
Therefore, 
\begin{equation*}
Rad(A)=lin\{e_{i}:i\in \Lambda _{Rad(A)}\}.
\end{equation*}
\end{corollary}

\begin{proof}
If $e_{i}^{2}\in Rad(A)$ if and only if $e_{i}^{2}\in M$ for every maximal
modular ideal $M$ of $A$. Then, by Theorem \ref{cuadrado} \ we have that $%
e_{i}^{2}\in M$ if and only if $e_{i}\in M.$ Consequently $e_{i}^{2}\in
Rad(A)$ if, and only if, $e_{i}\in Rad(A)$. The rest is clear.
\end{proof}

The Jacobson radical of a general non-associative algebra was studied in 
\cite{Ma-Ve2}. There, the following results\ (that we include here for the
sake of completeness)\ were proved.

\begin{proposition}
Let $A$ be an algebra and let $A_{1}$ be the unitization of $A.$ Then $%
Rad(A)=Rad(A_{1}).$ Consequently $A$ is semisimple if, and only if, $A_{1}$
is semisimple.
\end{proposition}

\begin{proposition}
If $A$ is a real algebra then, $Rad(A_{\mathbb{C}})\subseteq Rad(A)+iRad(A),$
where $A_{\mathbb{C}}$ denotes the complexification of $A.$ Therefore if $A$
is semisimple then, $A_{\mathbb{C}}$ is semisimple.
\end{proposition}

Given an evolution algebra $A,$ we define $%
a^{n}=a(a(a....(aa))))=L_{a}^{n-1}(a)$ where $L_{a}(b)=ab=ba,$ for every
\thinspace $a,b\in A.$

The set of all nilpotent elements of an evolution algebra is not necessarily
an ideal as the next example shows.

\begin{example}
Let $A$ be an evolution algebra with a natural basis $B:=\{e_{1},e_{2}\}$
such that $e_{1}^{2}=e_{2}$ and $e_{2}^{2}=e_{1}.$ Then $%
e_{1}^{3}=e_{1}e_{1}^{2}=e_{1}^{2}e_{1}=0$ and similarly $e_{2}^{3}=0.$ But $%
(e_{1}+e_{2})^{2}=e_{1}+e_{2}$ so that $(e_{1}+e_{2})^{n}=e_{1}+e_{2}.$
\end{example}

If an element $a$ of an evolution algebra $A$ is nilpotent then $a$ is
quasi-invertible. Indeed if $a^{n}=0$ then \ for $b=-(a+a^{2}+....+a^{n-1})$
we have that $a+b-ba=0$, so $b$ is a quasi-inverse of $a.$

In \cite[Proposition 9]{Ma-Ve2}~\ it was proved that if $Q$ is a
quasi-invertible ideal of a non-associative algebra then $Q\subseteq Rad(A).$
Therefore we obtain the following result.

\begin{proposition}
\label{nilpo}If $A$ is an evolution algebra then every quasi-invertible
ideal of $A$, and particularly every nilpotent ideal, is contained in $%
Rad(A).$
\end{proposition}

We point out that the radical of an evolution algebra may contain non
quasi-invertible elements, as we show next.

\begin{example}
\label{quasi-inver}\emph{Let }$A$\emph{\ be an evolution algebra with a
natural basis }$B:=\{e_{1},e_{2}\}$\emph{\ such that }$e_{1}^{2}=e_{1}$\emph{%
\ and }$e_{2}^{2}=e_{1}+e_{2}.$\emph{\ Then, by Corollary \ref{Jradical} we
have that }$Rad(A)=\mathbb{K}e_{1},$\emph{\ and }$e_{1}$\emph{\ is not
quasi-invertible. Note that the only non trivial modular ideal of }$A$\emph{%
\ is }$\mathbb{K}e_{1}$\emph{\ as it can be deduced from Corollary \ref%
{coroa}.}
\end{example}

$\hfill \square $

\section{Reviewing the notion of spectrum in the non-associative setting 
\label{espec}}

Since the beginning of the XX century, a great variety of non-associative
algebras have been used to model many phenomena in different scientific
contexts of Biology, Physics, or Engineering (see for instance \cite%
{Beck,Gulak,Okubo,Reed,Worz-Busekros}). Nevertheless, for non- associative
algebras, the notion of spectrum of an element was not considered before 
\cite{Ve} and \cite{Ma-Ve3}. Indeed the definition of an invertible element
in a non-associative algebra with a unit never was considered either.

Throughout this section $A$ will denote an algebra which does not need to be
associative or commutative. The classical notion of invertible element of an
associative algebra with a unit was extended to the non-associative setting
in the following two different ways, given in Definition \ref{definver1} and
Definition \ref{definver2} below.

\begin{definition}
\cite[Definition 2]{Ve}\label{definver1} \emph{Let }$A$\emph{\ be an algebra
with a unit }$e$\emph{. We say that }$a\in A$\emph{\ is \textbf{invertible}
if }$a$\emph{\ has a left and a right inverse (that is, there exists }$%
b,c\in A$\emph{\ such that }$bc=ca=e$\emph{). The set of invertible elements
of the algebra }$A$\emph{\ was denoted by }$inv(A).$
\end{definition}

The definition of spectrum related to the above definition of invertible
element is the next one (originally established in \cite[Definition 2]{Ve}).

\begin{definition}
\label{espe1}\emph{For a complex algebra }$A$\emph{\ with a unit} $e$\emph{,
the\textbf{\ spectrum} of }$a\in A$\emph{\ is defined as the set}%
\begin{equation*}
\sigma ^{A}(a):=\{\lambda \in \mathbb{C}:a-\lambda e\text{ is not invertible}%
\}.
\end{equation*}%
\emph{As in the associative case, the spectrum of an element }$a$ \emph{in a
complex algebra }$A$ \emph{without a unit \ is defined as }$\sigma
^{A_{1}}(a),$ \emph{that is the spectrum of }$a$ \emph{\ in }$A_{1}$ \emph{%
(the unitization of }$A$\emph{)}$.$\emph{\ Similarly the spectrum of an
element }$a$\emph{\ in a real algebra }$A$ \emph{is defined as }$\sigma ^{A_{%
\mathbb{C}}}(a),$ \emph{where }$A_{\mathbb{C}}$ \emph{is the
complexification of }$A.$
\end{definition}

To study the invertibility of an element in the unitization algebra $A_{1}$
in terms of the algebra $A$ \ we use the notion of quasi-invertible element
(as in the associative case). We recall that for an arbitrary algebra $A$%
\emph{\ }(associative or not)\ we say that $a\in A$\ is \textbf{left
quasi-invertible} if there exists $b\in A$\ such that $a+b-ab=0.$\ Similarly
it is said that $a$\ is \textbf{right quasi-invertible} if there exists $%
c\in A$\ such that\emph{\ }$a+c-ca=0.$\emph{\ }Finally we say that $a$\ is 
\textbf{quasi-invertible} if $a$\ is both left and right quasi-invertible.
We denote by $q-inv(A)$\ the set of quasi-invertible elements in $A.$

For an algebra $A$ without a unit, it was proved in \cite{Ve} that for every 
$a\in A,$

\begin{equation}
\sigma ^{A}(a):=\{0\}\cup \{\lambda \in \mathbb{C}\backslash \{0\}:\frac{a}{%
\lambda }\notin q-inv(A)\}.  \label{quasi}
\end{equation}%
Indeed, it is clear that if $A$ has no a unit then, $0\in \sigma ^{A}(a)$
for every $a\in A$ because $a\notin inv(A_{1}).$ Moreover if $\lambda \neq 0$
then $a-\lambda \mathbf{1}\in inv(A_{1})$ if and only if there exists $b\in
A $ such that $(a-\lambda \mathbf{1)(}b-\frac{1}{\lambda }\mathbf{1)=1}$,
which means that $\lambda b$ is a quasi-inverse of $\frac{a}{\lambda }.$
Thus $a-\lambda \mathbf{1}\in inv(A_{1})$ if and only if $\frac{a}{\lambda }%
\in q-inv(A).\,\ $Therefore we conclude that if $A$ is an algebra without a
unit then, 
\begin{equation}
\sigma ^{A}(a)\backslash \{0\}=\{\lambda \in \mathbb{C}\backslash \{0\}:%
\frac{a}{\lambda }\notin q-inv(A)\},  \label{quasi2}
\end{equation}%
for every $a\in A,$ which generalizes what happens whenever $A$ is
associative to the general setting that we are considering. If $A$ has a
unit, $e,$ a similar argument by replacing $1$ with $e$ shows us that (\ref%
{quasi2}) holds and, in this case, $0\in \sigma ^{A}(a)$ if and only if $%
a\notin inv(A).$

\bigskip

On the other hand, along the last century, non-associative division algebras
were considered to generalize the Gelfand-Mazur theorem among other goals
(see for instance \cite{Alb, Schafer,Ka1,Ka2}). Indeed, a main problem was
to look for fields, other than $\mathbb{C},$ to study their associated
geometries.

In these works, a \textbf{division algebra} was defined as an algebra $A$
with a unit (non necessarily associative) such that the left and right
multiplication operators, \thinspace $L_{a}$ and $R_{a},$ associated to
every non-zero element $a\in A$ are bijective. (We recall that $L_{a}$ \ and 
$R_{a}$ \ are, respectively, the operators on $A$ given by $L_{a}(b)=ab$ and 
$R_{a}(b)=ba,$ for every $b\in A$)$.$ In \cite{Ka2}, I. Kaplansky proved
that \textit{left division complete normed (non-associative) algebras are
isomorphic to} $\mathbb{C}$. Many years ago, I. Kaplansky had obtained
similar characterizations of the field of complex numbers in \cite{Ka1}.
Inspired by these facts, in \cite[Definition 2.1]{Ma-Ve3} (see also \cite%
{Ma-Ve4}) the following definition of an invertible element was established.
This definition is nothing but \cite[Proposition 1.19]{Bonsall} free of the
requirement of associativity.

\begin{definition}
\label{definver2}\emph{Let }$A$\emph{\ be an algebra with a unit. We say
that }$a\in A$\emph{\ is \textbf{multiplicatively invertible} (shortly 
\textbf{m-invertible}) if the corresponding multiplication operators }$L_{a}$%
\emph{\ and }$R_{a}$\emph{\ \ are bijective.}
\end{definition}

In relation with this definition of invertibility, the corresponding
definition of the spectrum that we establish next was originally introduced
in \cite[Definition 2.1]{Ma-Ve3}.

\begin{definition}
\emph{Let }$A$\emph{\ be a complex algebra with a unit. We define the\textbf{%
\ multiplicative spectrum,} or shortly the \textbf{m-spectrum}, of }$a\in A$%
\emph{\ as the set }%
\begin{equation*}
\sigma _{m}^{A}(a):=\{\lambda \in \mathbb{C}:a-\lambda e\text{ is not
m-invertible}\}.
\end{equation*}%
\emph{If }$A$\emph{\ is a complex algebra without a unit then, we define the
m-spectrum of }$a\in A$\emph{\ as the set }$\sigma _{m}^{A_{1}}(a)$\emph{\
where }$A_{1}\,\ $\emph{denotes the unitization of }$A.$\emph{\ Finally, if }%
$A$\emph{\ is a real algebra then, we define the m-spectrum of }$a\in A$%
\emph{\ as the set }$\sigma _{m}^{A_{\mathbb{C}}}(a)$\emph{\ where }$A_{%
\mathbb{C}}\,\ $\emph{denotes the complexification of }$A.$
\end{definition}

Therefore, for a complex algebra with a unit, the m-spectrum and the
spectrum of an element are related as follows%
\begin{equation*}
\sigma _{m}^{A}(a)=\sigma ^{\mathcal{L}(A)}(L_{a})\cup \sigma ^{\mathcal{L}%
(A)}(R_{a}),
\end{equation*}%
where $\mathcal{L}(A)$ denotes algebra of all linear operators on $A,$ which
is associative (even if $A$ is not it). Moreover from the Banach isomorphism
theorem, if $A$ is a complete complex normed algebra we have that 
\begin{equation*}
\sigma _{m}^{A}(a)=\sigma ^{L(A)}(L_{a})\cup \sigma ^{L(A)}(R_{a}),
\end{equation*}%
where $L(A)$ denotes the Banach algebra of all bounded linear operators on $%
A.$

As proved in\ \cite[Proposition 2.2]{Ma-Ve3}, for an arbitrary algebra $A$
and $a\in A,$ 
\begin{equation*}
\sigma _{m}^{A}(a)\backslash \{0\}=\sigma ^{\mathcal{L}(A)}(L_{a})\cup
\sigma ^{\mathcal{L}(A)}(R_{a}).
\end{equation*}%
In fact, if $A$ has not a unit then, 
\begin{equation}
\sigma _{m}^{A}(a)=\sigma ^{\mathcal{L}(A_{1})}(L_{a})\cup \sigma ^{\mathcal{%
L}(A_{1})}(R_{a})=\sigma ^{\mathcal{L}(A)}(L_{a})\cup \sigma ^{\mathcal{L}%
(A)}(R_{a})\cup \{0\},  \label{tres}
\end{equation}%
for every $a\in A,$ meanwhile if $A$ has a unit then, by definition, as said
already, 
\begin{equation*}
\sigma _{m}^{A}(a)=\sigma ^{\mathcal{L}(A)}(L_{a})\cup \sigma ^{\mathcal{L}%
(A)}(R_{a})\,.
\end{equation*}%
Moreover, as it is well known, if $A$ is an associative then,\ $\sigma
_{{}}^{A}(a)=\sigma _{m}^{A}(a),$ and this set is nothing but the classical
spectrum of $a$. However, for a non-associative algebra $A,$ and $a\in A,$
it turns out that $\sigma ^{A}(a)$ and $\sigma _{m}^{A}(a)$ may be
different. Of course, the following relation between the spectrum and the
m-spectrum is obvious 
\begin{equation*}
\sigma _{{}}^{A}(a)\subseteq \sigma _{m}^{A}(a).
\end{equation*}%
In the next example we show an element in an evolution algebra for which the
above inclusion is strict.

\begin{example}
\label{estrict}\emph{Let }$A$\emph{\ be the algebra generated by }$%
B:=\{e_{1},e_{2}\}$\emph{\ with structure matrix given by }%
\begin{equation*}
\left( 
\begin{array}{rr}
-\frac{1}{2}\smallskip \vspace{0.09in} & \text{ }\frac{3}{4} \\ 
-\frac{1}{3} & \frac{1}{2}%
\end{array}%
\right) .
\end{equation*}%
\emph{Let }$a=3e_{1}+2e_{2}.$\emph{\ If }$b=2e_{1}+\frac{4}{3}e_{2}$ \emph{%
then, }$ab=2e_{1}^{2}$ \emph{because} \emph{\ }%
\begin{equation*}
ab=6e_{1}^{2}+\frac{8}{3}e_{2}^{2}=-6(\frac{1}{2}e_{1}+\frac{1}{3}e_{2})+%
\frac{8}{3}(\frac{3}{4}e_{1}+\frac{1}{2}e_{2})=-3e_{1}-2e_{2}+2e_{1}+\frac{4%
}{3}e_{2}=-e_{1}-\frac{2}{3}e_{2}.
\end{equation*}%
\emph{Therefore }%
\begin{equation*}
-\frac{1}{2}a+b+\frac{1}{2}ab=-(\frac{3}{2}e_{1}+e_{2})+2e_{1}+\frac{4}{3}%
e_{2}+\frac{1}{2}(-e_{1}-\frac{2}{3}e_{2})=\allowbreak 0,
\end{equation*}%
\emph{so that }$-\frac{1}{2}\notin \sigma ^{A}(a).$\emph{\ However}%
\begin{equation*}
(L_{a}+\frac{1}{2}I)a=(3e_{1}+2e_{2})^{2}+\frac{1}{2}%
(3e_{1}+2e_{2})=9e_{1}^{2}+4e_{2}^{2}+\frac{1}{2}(3e_{1}+2e_{2})=
\end{equation*}%
\begin{equation*}
-9(\frac{1}{2}e_{1}+\frac{1}{3}e_{2})+4(\frac{3}{4}e_{1}+\frac{1}{2}e_{2})+%
\frac{1}{2}(3e_{1}+2e_{2})=0.
\end{equation*}%
\emph{which proves that }$L_{a}+\frac{1}{2}I$\emph{\ is not injective and
hence }$-\frac{1}{2}\in \sigma _{m}^{A}(a).$ \emph{Consequently }$\sigma
_{{}}^{A}(a)$\emph{\ is strictly contained into }$\sigma _{m}^{A}(a).$
\end{example}

Associated to each notion of spectrum we have the corresponding notion of
spectral radius.

\begin{definition}
\label{radius}\emph{Let }$A$\emph{\ be an algebra and let }$a\in A.$\emph{\
We define the \textbf{spectral radius} of }$a$\emph{\ as the value given by }%
\begin{equation*}
\rho (a):=\sup \{\left\vert \lambda \right\vert :\lambda \in \sigma
_{{}}^{A}(a)\},
\end{equation*}%
\emph{if }$\sigma _{{}}^{A}(a)$\emph{\ is non-empty and }$\rho (a):=\{0\}$%
\emph{\ otherwise. Thus, }$0\leq \rho (a)\leq \infty .$

\emph{Similarly we define the \textbf{m-spectral radius} of }$a$\emph{\ as
the value}%
\begin{equation*}
\rho _{m}(a):=\sup \{\left\vert \lambda \right\vert :\lambda \in \sigma
_{m}^{A}(a)\},
\end{equation*}%
\emph{if }$\sigma _{m}^{A}(a)$\emph{\ is non-empty, and }$\rho _{m}(a):=\{0\}
$\emph{\ otherwise (note that if }$A$\emph{\ does not need to be normed and
hence }$\sigma _{m}^{A}(a)$\emph{\ may be empty). }

\emph{Since }$\sigma _{{}}^{A}(a)\subseteq \sigma _{m}^{A}(a)$\emph{\ we
have that\ }$\rho (a)\leq \rho _{m}(a),$\emph{\ for every }$a\in A.$
\end{definition}

In the classical theory of Banach algebras the spectrum has a relevant role
in relation to the radical of the algebra. In fact, as is said in \cite[p.
189]{Palmer} (see also \cite[Theorem 4.3.6]{Dales}) \textit{the Jacobson
radical of an associative algebra }$A$ \textit{can be described as the
largest ideal on which the spectral radius of each element is identically
zero. Consequently, }$A$ \textit{is semisimple if its radical is zero, that
is if }$\{0\}$\textit{\ is the only ideal contained in the set of elements
having spectral radius equal to zero.}

According to this quote, since we have two natural notions of spectral
radius, we consider the following definitions of simplicity.

\begin{definition}
\label{semisimplicity}\emph{We say that an algebra }$A$\emph{\ is \textbf{%
spectrally semisimple }if zero is the unique ideal of }$A$\emph{\ contained
in the set }$\{a\in A:\rho (a)=0\}.$\emph{\ Similarly, we say that }$A$\emph{%
\ is \textbf{multiplicatively semisimple }or \textbf{m-semisimple }if \ zero
is the unique ideal of }$A$\emph{\ contained in the set }$\{a\in A:\rho
_{m}(a)=0\}.$
\end{definition}

If $A$ is associative then the two notions of semisimplicity defined above
coincide and mean precisely that the Jacobson radical of $A$ is zero (see 
\cite[Theorem 4.3.6]{Palmer}). Nevertheless for general algebras we have the
following.

\begin{proposition}
\label{rela}Let $A$ be an algebra, and consider the following assertions:

$\mathrm{(i)}$ $A$ is semisimple,

$\mathrm{(ii)}$ $A$ spectrally semisimple,

$\mathrm{(iii)}$ $A$ is m-semisimple.

Then $\mathrm{(i)}$ $\Rightarrow \mathrm{(ii)}\Rightarrow \mathrm{(iii)}.$
Moreover, if $A$ is associative then these assertions are equivalent.
\end{proposition}

\begin{proof}
$\mathrm{(i)}$ $\Rightarrow \mathrm{(ii).}$ If $A$ is not spectrally
semisimple then there exists a non-zero ideal $Q$ contained in $\{a\in
A:\rho (a)=0\}.$ Therefore $Q$ is quasi-invertible because $1\notin \sigma
_{{}}^{A}(a)$ for every $a\in A.$ By \cite[Proposition 9]{Ma-Ve2} we have
that $Q\subseteq Rad(A)$ and hence $A$ is not semisimple. The assertion $%
\mathrm{(ii)}$ $\Rightarrow \mathrm{(iii)}$ is obvious because $\rho
_{m}(a)=0~$implies that $\rho (a)=0,$ for $a\in A.\,$\ For the proof of the
fact \ that $\mathrm{(iii)}$ $\Rightarrow \mathrm{(i)}$ if $A$ is
associative, see for instance \cite[Theorem 4.3.6]{Palmer}.
\end{proof}

In \ \cite{Ma-Ve3} (see Theorem 3.5 and Corollary 3.6) we proved the
following result in the general non-associative setting.

\begin{theorem}
E\textit{every surjective homomorphism from a Banach algebra onto a
m-semisimple Banach algebra, is continuous}. Consequently, \textit{%
m-semisimple Banach algebras has a unique complete norm topology.}
\end{theorem}

Particularly, whenever $A$ is associative, we obtain as a corollary the well
known theorem of B. E. Johnson \cite{Johnson} (see also \cite{Bonsall,
Dales, Palmer}) that in words of T. Palmer is a "cornerstone of the Banach
algebra theory". Note that, in the above result, m-semisimple can be
replaced by spectrally semisimple or by semisimple (see \cite{Ma-Ve2}).

\section{The spectrum of an element in an evolution algebra}

In this section we characterize the spectrum and the m-spectrum of an
element in an evolution algebra. Moreover we study the notions of
semisimplicity established in the above section in the framework of
evolution algebras. Throughout this section all the algebras that we will
consider will be complex. This is not restrictive because the spectrum, as
well as the m-spectrum, of an element in a real algebra $A$ is defined as
the corresponding one in the complexified algebra $A_{\mathbb{C}}.$

Let $A$ be a finite dimensional evolution algebra and $B:=\{e_{1},...,e_{n}\}
$ a natural basis of $A.$ Let $M_{B}(A)=(\omega _{ij})\in M_{n\times n}(%
\mathbb{K})$ be the structure matrix of $A$ relative to $B.$ Then, it is
straightforward to check that for $a=\sum\limits_{i=1}^{n}\alpha _{i}e_{i}$
and $b=\sum\limits_{i=1}^{n}\beta _{i}e_{i}$ in $A$, we have 
\begin{eqnarray}
ab &=&\left( 
\begin{array}{ccc}
\omega _{11} & \cdots  & \omega _{1n} \\ 
\vdots  & \ddots  & \vdots  \\ 
\omega _{n1} & \cdots  & \omega _{nn}%
\end{array}%
\right) \left( 
\begin{array}{c}
\alpha _{1}\beta _{1} \\ 
\vdots  \\ 
\alpha _{n}\beta _{n}%
\end{array}%
\right) =  \label{producto} \\
&&\left( 
\begin{array}{ccc}
\omega _{11} & \cdots  & \omega _{1n} \\ 
\vdots  & \ddots  & \vdots  \\ 
\omega _{n1} & \cdots  & \omega _{nn}%
\end{array}%
\right) \left( 
\begin{array}{ccc}
\alpha _{1} & \cdots  & 0 \\ 
\vdots  & \ddots  & \vdots  \\ 
0 & \cdots  & \alpha _{n}%
\end{array}%
\right) \left( 
\begin{array}{c}
\beta _{1} \\ 
\vdots  \\ 
\beta _{n}%
\end{array}%
\right) .  \notag
\end{eqnarray}

The result can be easily adapted to the infinite dimensional case working
with the finite set of index defined by $\Lambda _{ab}$, i.e. the support of 
$ab.$ Note that if $\Lambda _{a}$ and $\Lambda _{b}$ denote as usual the
support of $a$ and $b$ (respectively) and if $\Lambda _{a}\cap \Lambda
_{b}\neq \emptyset $ then, $\Lambda _{ab}=D^{1}(\Lambda _{a}\cap \Lambda
_{b})$ (see Definition \ref{descen}).

We begin this section by determining the spectrum and the multiplicative
spectrum of an element in a non-zero trivial evolution algebra (in an
arbitrary algebra $A$ with zero product we have that $\sigma ^{A}(a)=\sigma
_{m}^{A}(a)=\{0\},$ for every $a\in A$). We recall that every structure
matrix of such an algebra is diagonal with non-zero entries (see Definition
2.8 and Remark \ref{ulti}). Moreover, by Proposition \ref{unidad}, a
non-zero trivial evolution algebra has a unit if, and only if, its dimension
is finite.

\begin{proposition}
\label{propoa}Let $A$ be a non-zero trivial evolution algebra, $%
B:=\{e_{i}:i\in \Lambda \}$ a natural basis$,$ and $M_{B}(A)=(\omega _{ij})$
the corresponding structure matrix$.$ Let $a\in A$ be such that $a=\sum
\alpha _{i}e_{i}.$ Then 
\begin{equation*}
\sigma ^{A}(a)=\sigma _{m}^{A}(a)=\{\alpha _{i}\omega _{ii}:i\in \Lambda \},
\end{equation*}%
and, 
\begin{equation*}
\sigma ^{A}(a)\backslash \{0\}=\sigma _{m}^{A}(a)\backslash \{0\}=\{\alpha
_{i}\omega _{ii}:i\in \Lambda _{a}\}.
\end{equation*}%
Moreover,

$\mathrm{(i)}$ If $\Lambda _{a}\neq \Lambda $ (this happens particularly
when $\dim A=\infty $) then $\sigma ^{A}(a)=\sigma _{m}^{A}(a)=\{\alpha
_{i}\omega _{ii}:i\in \Lambda _{a}\}\cup \{0\}.$

$\mathrm{(ii)}$ If $\Lambda _{a}=\Lambda $ $\ $(and hence $\dim A<\infty $)
then $\sigma ^{A}(a)=\sigma _{m}^{A}(a)=\{\alpha _{i}\omega _{ii}:i\in
\Lambda _{a}\}.$
\end{proposition}

\begin{proof}
If $A$ is a non-zero trivial evolution algebra and $B:=\{e_{i}:i\in \Lambda
\}$ is a natural basis then, $e_{i}^{2}=\omega _{ii}e_{i}$ with $\omega
_{ii}\neq 0\,,$ for every $i\in \Lambda .$ Let $a=\sum \alpha _{i}e_{i}\in A.
$ We have that $\frac{a}{\lambda }$ $\in q-inv(A)$ for $\lambda \neq 0$ if
and only if there exists $b\in A$ such that $a+\lambda b-ab=0.$ If $b=\sum
\beta _{i}e_{i}$ \ this means that $\sum \alpha _{i}e_{i}+\lambda \sum \beta
_{i}e_{i}-\sum \alpha _{i}\beta _{i}\omega _{ii}e_{i}=0,$ so that $\alpha
_{i}+(\lambda -\alpha _{i}\omega _{ii})\beta _{i}=0,$ for very $i\in \Lambda
.$ Consequently $\frac{a}{\lambda }$ $\notin q-inv(A)$ if, and only if, $%
\lambda -\alpha _{i}\omega _{ii}=0$ and $\alpha _{i}\neq 0.$ Since by (\ref%
{quasi2}) 
\begin{equation*}
\sigma ^{A}(a)\backslash \{0\}=\{\lambda \in \mathbb{C}\backslash \{0\}:%
\frac{a}{\lambda }\notin q-inv(A)\},
\end{equation*}%
we obtain that $\ \sigma ^{A}(a)\backslash \{0\}=\{\alpha _{i}\omega
_{ii}:i\in \Lambda _{a}\}.$ Moreover, by Proposition \ref{unidad}, the
non-zero trivial evolution algebra $A$ has a unit if, and only if, $\dim
A<\infty ,$ and in this case $a\in inv(A)$ (that is $0\notin \sigma ^{A}(a)$%
) if and only if $\Lambda _{a}=\Lambda .$ Thus $0\in \sigma ^{A}(a)$ if, and
only if, $\Lambda _{a}\neq \Lambda .\,\ $

With respect to the $m-$invertibility, by (\ref{tres}), we have $\sigma
_{m}^{A}(a)\backslash \{0\}=\sigma ^{\mathcal{L}(A)}(L_{a}).$ For $\lambda
\neq 0,$ note that $L_{a}-\lambda I$ $\ $\ is injective (respectively
surjective) if, and only if, $\alpha _{i}\omega _{ii}-\lambda \neq 0,$ for
every $i\in \Lambda _{a}.$ Therefore $\sigma _{m}^{A}(a)\backslash
\{0\}=\{\alpha _{i}\omega _{ii}:i\in \Lambda _{a}\}.$ Moreover, if $A$ has
no unit then, $0\in \sigma _{m}^{A}(a)$ as showed in \cite[Proposition 2.2]%
{Ma-Ve3}. If $A$ has a unit then $\dim A<\infty ,$ so that $L_{a}$ is
bijective if, and only if, $L_{a\text{ }}$ is injective which happens if,
and only if, $\Lambda _{a}=\Lambda .$ Thus $0\in \sigma _{m}^{A}(a)$ if, and
only if, $\Lambda _{a}\neq \Lambda ,$ and the result follows.
\end{proof}

\begin{corollary}
\label{unodime}One-dimensional evolution algebras with non-zero product are
m-semisimple (and hence spectrally semisimple and semisimple).
\end{corollary}

\begin{proof}
If $A=\mathbb{K}e$ with $e^{2}=\omega e$ for $\omega \neq 0$ then, we have
that $\sigma ^{A}(e)=\sigma _{m}^{A}(e)=\{\omega \}$ and the result follows.
\end{proof}

For a non-trivial evolution algebra $A,$ we have the following
characterization of the m-spectrum, $\sigma _{m}^{A}(a),$ and the spectrum, $%
\sigma ^{A}(a),$ of an element $a\in A.$

\begin{proposition}
\label{propob}Let $A$ be a finite-dimensional non-trivial evolution algebra
and $B:=\{e_{1},...,e_{n}\}$ a natural basis. If $\lambda \in \mathbb{C},$
and if $\ a=\sum_{i=1}^{n}\alpha _{i}e_{i}$ then,

$\mathrm{(i)}$ $\lambda \in \sigma _{m}^{A}(a)$ if, and only if, $\lambda =0$
or $\lambda $ is an eigenvalue of the matrix 
\begin{equation*}
\left( 
\begin{array}{ccc}
\omega _{11} & \cdots & \omega _{1n} \\ 
\vdots & \ddots & \vdots \\ 
\omega _{n1} & \cdots & \omega _{nn}%
\end{array}%
\right) \left( 
\begin{array}{ccc}
\alpha _{1} & \cdots & 0 \\ 
\vdots & \ddots & \vdots \\ 
0 & \cdots & \alpha _{n}%
\end{array}%
\right) .
\end{equation*}

$\mathrm{(ii)}$ $\lambda \in \sigma ^{A}(a)$ if, and only if, $\lambda =0$
or the equation 
\begin{equation*}
\left( \left( 
\begin{array}{ccc}
\omega _{11} & \cdots & \omega _{1n} \\ 
\vdots & \ddots & \vdots \\ 
\omega _{n1} & \cdots & \omega _{nn}%
\end{array}%
\right) \left( 
\begin{array}{ccc}
\alpha _{1} & \cdots & 0 \\ 
\vdots & \ddots & \vdots \\ 
0 & \cdots & \alpha _{n}%
\end{array}%
\right) -\left( 
\begin{array}{ccc}
\lambda & \cdots & 0 \\ 
\vdots & \ddots & \vdots \\ 
0 & \cdots & \lambda%
\end{array}%
\right) \right) \left( 
\begin{array}{c}
\beta _{1} \\ 
\vdots \\ 
\beta _{n}%
\end{array}%
\right) =\left( 
\begin{array}{c}
\alpha _{1} \\ 
\vdots \\ 
\alpha _{n}%
\end{array}%
\right)
\end{equation*}%
$%
\normalsize%
$has no solution for $\lambda $ (in which case $\lambda \in \sigma
_{m}^{A}(a)$).
\end{proposition}

\begin{proof}
By Corollary \ref{anterior} we have that $A$ has not a unit, and by (\ref%
{tres}),%
\begin{equation*}
\sigma _{m}^{A}(a):=\{0\}\cup \sigma ^{\mathcal{L}(A)}(L_{a}).
\end{equation*}%
Since $(L_{a}-\lambda I)$ is bijective if, and only if, it is injective
(because $A$ is finite-dimensional)\ we have from (\ref{producto}) that $%
\lambda \in \sigma ^{\mathcal{L}(A)}(L_{a})$ if, and only if, the equation 
\begin{equation*}
\left( \left( 
\begin{array}{ccc}
\omega _{11} & \cdots & 0 \\ 
\vdots & \ddots & \vdots \\ 
0 & \cdots & \omega _{nn}%
\end{array}%
\right) \left( 
\begin{array}{ccc}
\alpha _{1} & \cdots & 0 \\ 
\vdots & \ddots & \vdots \\ 
0 & \cdots & \alpha _{n}%
\end{array}%
\right) -\left( 
\begin{array}{ccc}
\lambda & \cdots & 0 \\ 
\vdots & \ddots & \vdots \\ 
0 & \cdots & \lambda%
\end{array}%
\right) \right) \left( 
\begin{array}{c}
\beta _{1} \\ 
\vdots \\ 
\beta _{n}%
\end{array}%
\right) =\left( 
\begin{array}{c}
0 \\ 
\vdots \\ 
0%
\end{array}%
\right)
\end{equation*}%
has a non-zero solution $b=\sum_{i=1}^{n}\beta _{i}e_{i}$ which proves $%
\mathrm{(i)}.$ Similarly, by (\ref{quasi}) we have that 
\begin{equation*}
\sigma ^{A}(a):=\{0\}\cup \{\lambda \in \mathbb{C}\backslash \{0\}:\frac{a}{%
\lambda }\notin q-inv(A)\}.
\end{equation*}%
Take $\lambda \in \mathbb{C}\backslash \{0\}.$ Since $\frac{a}{\lambda }\in
q-inv(A)$ if, and only if, there exists $b\in A$ such that $\frac{a}{\lambda 
}+b-\frac{a}{\lambda }b=0,$ or equivalently $ab-\lambda b=a,$ we obtain from
(\ref{producto}) that $\frac{a}{\lambda }\notin q-inv(A)$ if, and only if,
the equation\ 
\begin{equation*}
\left( \left( 
\begin{array}{ccc}
\omega _{11} & \cdots & \omega _{1n} \\ 
\vdots & \ddots & \vdots \\ 
\omega _{n1} & \cdots & \omega _{nn}%
\end{array}%
\right) \left( 
\begin{array}{ccc}
\alpha _{1} & \cdots & 0 \\ 
\vdots & \ddots & \vdots \\ 
0 & \cdots & \alpha _{n}%
\end{array}%
\right) -\left( 
\begin{array}{ccc}
\lambda & \cdots & 0 \\ 
\vdots & \ddots & \vdots \\ 
0 & \cdots & \lambda%
\end{array}%
\right) \right) \left( 
\begin{array}{c}
\beta _{1} \\ 
\vdots \\ 
\beta _{n}%
\end{array}%
\right) =\left( 
\begin{array}{c}
\alpha _{1} \\ 
\vdots \\ 
\alpha _{n}%
\end{array}%
\right)
\end{equation*}%
has no solution $b=\sum_{i=1}^{n}\beta _{i}e_{i}$ (which also implies that $%
\lambda \in \sigma ^{\mathcal{L}(A)}(L_{a})$). This proves $\mathrm{(ii)}.$
\end{proof}

\begin{example}
\label{exafinal}\emph{The evolution algebra }$A$\emph{\ given in Example \ref%
{estrict}, whose structure matrix with respect to }$B:=\{e_{1},e_{2}\}$\emph{%
\ is }%
\begin{equation*}
\left( 
\begin{array}{rr}
-\frac{1}{2}\smallskip \vspace{0.09in} & \text{ }\frac{3}{4} \\ 
-\frac{1}{3} & \text{ }\frac{1}{2}%
\end{array}%
\right) ,
\end{equation*}%
\emph{is such that }$e_{1}^{2}=-\frac{3}{2}e_{2}^{2},$\emph{\ so that the
ideal generated by }$e_{1}^{2}$\emph{\ is }$\left\langle
e_{1}^{2}\right\rangle =\mathbb{K}e_{1}^{2}$\emph{\ .\thinspace\ Similarly,
if }$a$\emph{\ }$\in A\backslash \{0\}$\emph{\ is not \ multiple of }$%
e_{1}^{2}$\emph{\ then }$\left\langle a\right\rangle =A$\emph{\ (because }$%
\dim \left\langle a\right\rangle =2$\emph{) and hence the unique non-zero
proper ideal of }$A$\emph{\ is }$\mathbb{K}e_{1}^{2}$\emph{\ . Since the
eigenvalues of }%
\begin{equation*}
\left( 
\begin{array}{rr}
-\frac{1}{2}\smallskip \vspace{0.09in} & \frac{3}{4} \\ 
-\frac{1}{3} & \frac{1}{2}%
\end{array}%
\right) \left( 
\begin{array}{rr}
-\frac{1}{2}\smallskip \vspace{0.09in} & 0 \\ 
0 & -\frac{1}{3}%
\end{array}%
\right) =\left( 
\begin{array}{cc}
\frac{1}{4}\smallskip \vspace{0.09in}\text{ } & -\frac{1}{4} \\ 
\frac{1}{6} & -\frac{1}{6}%
\end{array}%
\right) 
\end{equation*}%
\emph{are }$0$\emph{\ and }$\frac{1}{12},$\emph{\ }$\ $\emph{we conclude
that }$\sigma _{m}^{A}(e_{1}^{2})=\{0,\frac{1}{12}\}.$\emph{\ On the other
hand, }%
\begin{equation*}
\left[ \left( 
\begin{array}{cc}
\frac{1}{4}\smallskip \vspace{0.09in}\text{ } & -\frac{1}{4} \\ 
\frac{1}{6} & -\frac{1}{6}%
\end{array}%
\right) -\left( 
\begin{array}{cc}
\frac{1}{12}\smallskip \vspace{0.09in}\text{ } & 0 \\ 
0 & \frac{1}{12}%
\end{array}%
\right) \right] \left( 
\begin{array}{c}
\beta _{1}\smallskip  \\ 
\beta _{2}%
\end{array}%
\right) =\left( 
\begin{array}{c}
-\frac{1}{2}\smallskip \vspace{0.09in} \\ 
-\frac{1}{3}%
\end{array}%
\right) 
\end{equation*}%
\emph{has no solution as}%
\begin{equation*}
\left( 
\begin{array}{c}
\frac{1}{6}\beta _{1}\smallskip \text{ }\smallskip \vspace{0.09in}-\frac{1}{4%
}\beta _{2} \\ 
\frac{1}{6}\beta _{1}-\frac{1}{4}\beta _{2}%
\end{array}%
\right) \neq \left( 
\begin{array}{c}
-\frac{1}{2}\smallskip \vspace{0.09in} \\ 
-\frac{1}{3}%
\end{array}%
\right) ,
\end{equation*}%
\emph{so that, from the above theorem,}%
\begin{equation*}
\sigma ^{A}(e_{1}^{2})=\sigma _{m}^{A}(e_{1}^{2})=\{0,\frac{1}{12}\}.
\end{equation*}%
\emph{It follows that }$\{0\}$\emph{\ is the unique ideal contained in }$%
\{a\in A:\rho _{m}(a)=0\},$\emph{\ (and therefore also in }$\{a\in A:\rho
(a)=0\}$\emph{)}$,$\emph{\ so that }$A$\emph{\ is m-semisimple (and
spectrally semisimple). Nevertheless }$A$\emph{\ is a radical algebra
because }$A$\emph{\ has no modular ideals as it can be deduced from
Corollary \ref{Jradical}.}$\ $\emph{Actually }$e_{1}-e_{1}u\notin
\left\langle e_{1}^{2}\right\rangle =\mathbb{K}e_{1}^{2}$\emph{\ }\ \emph{%
for any} $u\in A.$
\end{example}

Next we determine the spectrum and the m-spectrum of an element in an
evolution algebra of arbitrary dimension.

\begin{proposition}
\label{propoc}Let $A$ be a non trivial evolution algebra, and $%
B:=\{e_{i}:i\in \Lambda \}$ a natural basis. For $a=\dsum\limits_{i\in
\Lambda _{a}}\alpha _{i}e_{i}\in A$, define%
\begin{equation*}
\begin{array}{l}
B^{0}(a):=\{e_{i}:i\in \Lambda _{a}\}. \\ 
B^{1}(a):=\cup _{i\in \Lambda _{a}}\{e_{j}:\omega _{ji}\neq 0\}.%
\end{array}%
\end{equation*}

If $B^{0}(a)=\{e_{1},\cdots ,e_{k}\}$ and $B^{0}(a)\cup
B^{1}(a)=\{e_{1},\cdots ,e_{k},e_{k+1},\cdots ,e_{m}\}$ then,

$\mathrm{(i)}$ $\lambda \in \sigma _{m}^{A}(a)$ if, and only if, $\lambda =0$
or $\lambda $ is an eigenvalue of the matrix 
\begin{equation*}
\left( 
\begin{array}{ccc}
\omega _{11} & \cdots & \omega _{1m} \\ 
\vdots & \ddots & \vdots \\ 
\omega _{m1} & \cdots & \omega _{mm}%
\end{array}%
\right) \left( 
\begin{array}{ccc}
\alpha _{1} & \cdots & 0 \\ 
\vdots & \ddots & \vdots \\ 
0 & \cdots & \alpha _{m}%
\end{array}%
\right) ,
\end{equation*}%
where \ $\alpha _{k+1}=\cdots =\alpha _{m}=0.$

$\mathrm{(ii)}$ $\lambda \in \sigma ^{A}(a)$ if, and only if, $\lambda =0$
or the equation 
\begin{equation*}
\tiny%
\left( \left( 
\begin{array}{ccc}
\omega _{11} & \cdots & \omega _{1m} \\ 
\vdots & \ddots & \vdots \\ 
\omega _{m1} & \cdots & \omega _{mm}%
\end{array}%
\right) \left( 
\begin{array}{ccc}
\alpha _{1} & \cdots & 0 \\ 
\vdots & \ddots & \vdots \\ 
0 & \cdots & \alpha _{m}%
\end{array}%
\right) -\left( 
\begin{array}{ccc}
\lambda & \cdots & 0 \\ 
\vdots & \ddots & \vdots \\ 
0 & \cdots & \lambda%
\end{array}%
\right) \right) \left( 
\begin{array}{c}
\beta _{1} \\ 
\vdots \\ 
\beta _{m}%
\end{array}%
\right) {\tiny =}\left( 
\begin{array}{c}
\gamma _{1} \\ 
\vdots \\ 
\gamma _{m}%
\end{array}%
\right) {\tiny ,}
\end{equation*}%
$%
\normalsize%
$where $\alpha _{k+1}=\cdots =\alpha _{m}=0,$ has no solution for $\lambda $
(in which case $\lambda \in \sigma _{m}^{A}(a)$).
\end{proposition}

\begin{proof}
Since $A$ is a non trivial evolution algebra, by Proposition \ref{unidad},
we have that $A$ has not a unit and hence $0\in \sigma ^{A}(a)\subseteq
\sigma _{m}^{A}(a).$ On the other hand, let $A_{0}:=lin\{e_{1},\cdots
,e_{k},e_{k+1},\cdots ,e_{m}\}.$ If $c\in A$ then, there exists a unique $%
c_{0}\in A_{0}$ and $c_{1}\in lin(B\backslash \{e_{1},\cdots
,e_{k},e_{k+1},\cdots ,e_{m}\})$ such that $c=c_{0}+c_{1}.$ We claim that
for $\lambda \in \mathbb{C}\backslash \{0\}$ and $c\in A,$ the equation $%
(L_{a}-\lambda I)b=c$ has a unique solution, $b\in A$ if, and only if, the
equation $(L_{a}-\lambda I)b_{0}=c_{0\text{ }}$ has a unique solution $\
b_{0}\in A_{0}.$ In fact, if $b=b_{0}+b_{1}$ and $c=c_{0}+c_{1}$ then $%
ab=ab_{0}\in A_{0}$ so that $ab-\lambda b=c$ if, and only if, $%
ab_{0}-\lambda b_{0}=c_{0}$ and \ $-\lambda b_{1}=c_{1}.$ Consequently, the
claim is proved because the necessary condition is obvious and, conversely,
given \ $(L_{a}-\lambda I)b=c,$ if $b_{0}\in A$ is the unique solution of $%
(L_{a}-\lambda I)b_{0}=c_{0\text{ }}$then $b=b_{0}+b_{1}$ with $b_{1}=-\frac{%
1}{\lambda }c_{1}$ is the unique solution of\ $(L_{a}-\lambda I)b=c.$ The
result follows directly from this fact. Indeed, if $b_{0}=\sum_{i=1}^{k}%
\beta _{i}e_{i}+\sum_{i=k+1}^{m}\beta _{i}e_{i}\in A_{0}$ then, 
\begin{equation*}
ab=\sum_{i=1}^{k}\alpha _{i}\beta _{i}e_{i}^{2}=\sum_{i=1}^{m}\eta _{i}e_{i},
\end{equation*}%
where 
\begin{equation*}
\left( 
\begin{array}{l}
\eta _{1} \\ 
\vdots \\ 
\vdots \\ 
\vdots \\ 
\eta _{m}%
\end{array}%
\right) =\left( 
\begin{array}{llll}
\omega _{11} & \omega _{12} & \cdots & \omega _{1k} \\ 
\vdots & \vdots &  &  \\ 
\omega _{k1} & \omega _{k2} & \cdots & \omega _{kk} \\ 
\vdots & \vdots &  &  \\ 
\omega _{m1} & \omega _{m2} & \cdots & \omega _{mk}%
\end{array}%
\right) \left( 
\begin{array}{l}
\alpha _{1}\beta _{1} \\ 
\vdots \\ 
\vdots \\ 
\vdots \\ 
\alpha _{k}\beta _{k}%
\end{array}%
\right) .
\end{equation*}%
$%
\noindent%
$Therefore $(L_{a}-\lambda I)$ is bijective if and only if the equation 
\begin{equation}
\tiny%
\left( 
\begin{tabular}{llllll}
$\omega _{11}$ & $\cdots $ & $\omega _{1k}$ & $\cdots $ & $\cdots $ & $%
\omega _{1m}$ \\ 
$\vdots $ & $\ddots $ & $\vdots $ &  &  & $\vdots $ \\ 
$\omega _{k1}$ & $\cdots $ & $\omega _{kk}$ &  &  & $\vdots $ \\ 
$\vdots $ &  &  & $\ddots $ &  & $\vdots $ \\ 
$\vdots $ &  &  &  & $\ddots $ & $\vdots $ \\ 
$\omega _{m1}$ & $\cdots $ & $\cdots $ & $\cdots $ & $\cdots $ & $\omega
_{mm}$%
\end{tabular}%
\right) \left( 
\begin{tabular}{l}
$\alpha _{1}\beta _{1}$ \\ 
$\vdots $ \\ 
$\alpha _{k}\beta _{k}$ \\ 
$0$ \\ 
$\vdots $ \\ 
$0$%
\end{tabular}%
\right) {\tiny -}\left( 
\begin{tabular}{l}
$\lambda \beta _{1}$ \\ 
$\vdots $ \\ 
$\lambda \beta _{k}$ \\ 
$\lambda \beta _{k+1}$ \\ 
$\vdots $ \\ 
$\lambda \beta _{m}$%
\end{tabular}%
\right) {\tiny =}\left( 
\begin{tabular}{l}
$\gamma _{1}$ \\ 
$\vdots $ \\ 
$\gamma _{k}$ \\ 
$\gamma _{k+1}$ \\ 
$\vdots $ \\ 
$\gamma _{m}$%
\end{tabular}%
\right)  \label{arri}
\end{equation}%
has a unique solution $b_{0}=\dsum\limits_{i=1}^{i=m}\beta _{i}e_{i}$ for
every $c_{0}=\dsum\limits_{i=1}^{i=m}\gamma e_{i}.$ This means that, for $%
\alpha _{k+1}=\cdots =\alpha _{m}=0,$ the equation 
\begin{equation*}
\tiny%
\left( \left( 
\begin{array}{ccc}
\omega _{11} & \cdots & \omega _{1m} \\ 
\vdots & \ddots & \vdots \\ 
\omega _{m1} & \cdots & \omega _{mm}%
\end{array}%
\right) \left( 
\begin{array}{ccc}
\alpha _{1} & \cdots & 0 \\ 
\vdots & \ddots & \vdots \\ 
0 & \cdots & \alpha _{m}%
\end{array}%
\right) -\left( 
\begin{array}{ccc}
\lambda & \cdots & 0 \\ 
\vdots & \ddots & \vdots \\ 
0 & \cdots & \lambda%
\end{array}%
\right) \right) \left( 
\begin{array}{c}
\beta _{1} \\ 
\vdots \\ 
\beta _{m}%
\end{array}%
\right) {\tiny =}\left( 
\begin{array}{c}
\gamma _{1} \\ 
\vdots \\ 
\gamma _{m}%
\end{array}%
\right) {\tiny ,}
\end{equation*}%
has a unique solution. Equivalently $(L_{a}-\lambda I)$ is not bijective if,
and only if, that $\lambda $ is an eigenvalue of the matrix 
\begin{equation*}
\left( 
\begin{array}{ccc}
\omega _{11} & \cdots & \omega _{1m} \\ 
\vdots & \ddots & \vdots \\ 
\omega _{m1} & \cdots & \omega _{mm}%
\end{array}%
\right) \left( 
\begin{array}{ccc}
\alpha _{1} & \cdots & 0 \\ 
\vdots & \ddots & \vdots \\ 
0 & \cdots & \alpha _{m}%
\end{array}%
\right) ,
\end{equation*}%
where $\alpha _{k+1}=\cdots =\alpha _{m}=0.$ This proves $\mathrm{(i)}$.

On the other hand, $\frac{a}{\lambda }\in q-inv(A)$ if and only if the
equation (\ref{arri}) has a solution when $\gamma _{i}=$ $\alpha _{i},$ for $%
i=1,\cdots ,k,$ and $\gamma _{k+1}=\cdots =\gamma _{m}=0.$ This proves $%
\mathrm{(ii)}$.
\end{proof}

\bigskip For the next result, recall that if $B:=\{e_{i}:i\in \Lambda \}$ is
a natural basis of an evolution algebra $A,$ and if $\Lambda _{0}\subseteq
\Lambda $ is non-empty, then set of descendents of $\Lambda _{0}$ is defined
as%
\begin{equation*}
D(\Lambda _{0})=\dbigcup\limits_{i\in \Lambda _{0}}D(i),
\end{equation*}%
where $D(i)$ denotes the set of descendents of $i$ (see Definition \ref%
{descen}).

\begin{proposition}
\label{segununo}Let $A$ be an evolution algebra and $B:=\{e_{i}:i\in \Lambda
\}$ a natural basis$.$ Let $J$ be an ideal of $A,$with support $\Lambda _{J}.
$ Then, $I_{1}:=lin\{e_{i}^{2}:i\in \Lambda _{J}\cup D(\Lambda _{J})\}$ and $%
I_{2}:=lin\{e_{i}:i\in \Lambda _{J}\cup D(\Lambda _{J})\}$ are ideals of $A$
and 
\begin{equation*}
I_{1}\subseteq J\subseteq I_{2}.
\end{equation*}%
Moreover, if $\dim A<\infty $ and $\det M_{B}(A)\neq 0$ then, $I=J=I_{2}$.
\end{proposition}

\begin{proof}
That $I_{1}$ is an ideal is clear because if $i\in \Lambda _{J}\cup
D(\Lambda _{J})$ and $j\in \Lambda $ then, either $e_{j}e_{i}^{2}=0$ or $%
e_{j}e_{i}^{2}=\omega _{ji}e_{j}^{2}$ with $\omega _{ji}\neq 0$ so that $%
j\in D^{1}(i)\subseteq D(\Lambda _{J}).\,\ $Similarly, $I_{2}$ is another
ideal because if $i\in \Lambda _{J}\cup D(\Lambda _{J})$ then $%
e_{i}^{2}=\sum_{k\in \Lambda }\omega _{ki}e_{k}$ and $\Lambda
_{e_{i}^{2}}\subseteq D(\Lambda _{J}).$ Obviously $I_{1}\subseteq J\subseteq
I_{2}$. On the other hand, if $\dim A<\infty $ and $\det M_{B}(A)\neq 0$
then $I_{1}=J=I_{2}$ as $\dim I_{1}=\dim I_{2}.$
\end{proof}

\begin{corollary}
\label{otro}Let $A$ be an evolution algebra and $B:=\{e_{i}:i\in \Lambda \}$
a natural basis of $A.$ Then, \ for every $i\in \Lambda ,$ the ideal
generated by $e_{i}^{2}$ is 
\begin{equation*}
\left\langle e_{i}^{2}\right\rangle =lin\{e_{j}^{2}:j\in D(i)\cup \{i\}\}.
\end{equation*}%
Consequently the dimension of every one-generated ideal in an evolution
algebra is countable.
\end{corollary}

\begin{proof}
From the above result we have that $lin\{e_{j}^{2}:j\in D(i)\cup \{i\}\}$ is
an ideal contained into $\left\langle e_{i}^{2}\right\rangle ,$ and
obviously $e_{i}^{2}\in lin\{e_{j}^{2}:j\in D(i)\cup \{i\}\},$ so the result
follows (note that $D(i)$ is countable).
\end{proof}

In the next result we characterize m-semisimple evolution algebras with
finite dimension.

\begin{corollary}
Let $A$ be an evolution algebra and $B:=\{e_{i}:i\in \Lambda \}$ a natural
basis.

$\mathrm{(i)}\ A$ is spectrally semisimple if and only if, for every index $%
i\in \Lambda $ there exists $a$ in $lin\{e_{j}^{2}:j\in D(i)\cup \{i\}\}$
such that $\sigma ^{A}(a)\neq 0.$

$\mathrm{(ii)}\ A$ is m-semisimple if, and only if, for every index $i\in
\Lambda $ there exists $a$ in $lin\{e_{j}^{2}:j\in D(i)\cup \{i\}\}$ such
that $\sigma _{m}^{A}(a)\neq 0.$
\end{corollary}

\begin{proof}
By definition, $A$ is spectrally semisimple (respectively m-semisimple)\ if
and only if, the set $\{a\in A:\sigma ^{A}(a)=0\}$ (respectively $\{a\in
A:\sigma _{m}^{A}(a)=0\}$) does not contain a non-zero ideal. Since a subset 
$S\subseteq A$ contains a non-zero ideal if and only if $S$ contains an
ideal of the type $\left\langle e_{i}^{2}\right\rangle $ for some $i\in
\Lambda ,$ the result follows from Corollary \ref{otro}.
\end{proof}

For finite dimensional evolution algebras we have the following
characterization of the m-semisimplicity.

\begin{corollary}
Let $A$ be a finite-dimensional evolution algebra and $B:=\{e_{1},...,e_{n}\}
$ a natural basis. Then $A$ is m-semisimple if, and only if, for every $%
i=1,...,n,$ there exists $a=\sum_{k=1}^{n}\alpha _{k}e_{k}\in
lin\{e_{j}^{2}:j\in D(i)\cup \{i\}\}$ such that the matrix 
\begin{equation*}
M_{B}(A)=\left( 
\begin{array}{ccc}
\omega _{11} & \cdots  & \omega _{1n} \\ 
\vdots  & \ddots  & \vdots  \\ 
\omega _{n1} & \cdots  & \omega _{nn}%
\end{array}%
\right) \left( 
\begin{array}{ccc}
\alpha _{1} & \cdots  & 0 \\ 
\vdots  & \ddots  & \vdots  \\ 
0 & \cdots  & \alpha _{n}%
\end{array}%
\right) 
\end{equation*}%
has a non-zero eigenvalue.
\end{corollary}

\begin{proof}
The result follows directly from the above corollary and Proposition \ref%
{propob}.
\end{proof}

The following result provides a helpful sufficient condition for the
semisimplicity of a finite dimensional evolution algebra.

\begin{corollary}
\label{gene1}Let $A$ be an evolution algebra, and $B:=\{e_{1},...,e_{n}\}$ a
natural basis. If, for every $i=1,...,n,$ the matrix 
\begin{equation*}
M_{j}(B)=\left( 
\begin{array}{ccc}
\omega _{11} & \cdots  & \omega _{1n} \\ 
\vdots  & \ddots  & \vdots  \\ 
\omega _{n1} & \cdots  & \omega _{nn}%
\end{array}%
\right) \left( 
\begin{array}{ccc}
\omega _{i1} & \cdots  & 0 \\ 
\vdots  & \ddots  & \vdots  \\ 
0 & \cdots  & \omega _{in}%
\end{array}%
\right) 
\end{equation*}%
has a non-zero eigenvalue then, $A$ is m-semisimple.
\end{corollary}

\begin{proof}
The proof is clear as $\sigma _{m}^{A}(e_{i}^{2})\neq 0$ for every $%
i=1,...,n.$ Therefore zero is the unique ideal contained into $\{a\in
A:\sigma _{m}^{A}(a)=0\}$ .
\end{proof}

\end{document}